\documentclass{amsart}
\usepackage{graphicx}
\usepackage{graphicx,color}
\usepackage{latexsym}
\usepackage{hyperref}
\usepackage{amsfonts}
\usepackage[all]{xy}
\usepackage[notref,notcite]{}
\usepackage{amssymb, mathrsfs, amsfonts, amsmath}

\usepackage{enumerate}

\usepackage{eulervm}
\usepackage{geometry}
\providecommand{\U}[1]{\protect\rule{.1in}{.1in}}

\newtheorem{theorem}{Theorem}
\newtheorem{acknowledgement*}{Acknowledgement}

\newtheorem{corollary}[theorem]{Corollary}

\newtheorem{definition}[theorem]{Definition}
\newtheorem{example}[theorem]{Example}

\newtheorem{lemma}[theorem]{Lemma}

\newtheorem{proposition}[theorem]{Proposition}
\newtheorem{remark}[theorem]{Remark}

\newcommand{\rr}{\mathbb{R}}

\newcommand{\hh}{\mathbb{H}}
\newcommand{\cb}{\mathcal{B}}
\newcommand{\sect}{\operatorname{Sect}}

\newcommand{\intM}{\mathrm{int}\, M}

\newcommand{\N}{\mathcal{N}}
\newcommand{\D}{\mathcal{D}}
\newcommand{\G}{\text{ }\!^{\D}\!G}

\newcommand{\inte}{\mathrm{int}}

\newcommand{\pma}{\partial_{\mathcal{M}}}
\newcommand{\CSH}{\mathcal{SH}}
\newcommand{\CH}{\mathcal{H}}

\begin{document}

\title[Dirichlet parabolicity and $L^1$-Liouville Property]{Dirichlet parabolicity and $L^1$-Liouville property under localized geometric conditions}

{\author{Leandro F. Pessoa${}^*$}
\address{Departamento de Matem\'{a}tica\\Universidade Federal do Piau\'{i}-UFPI\\
64049-550, Teresina-Brazil} \email{leandropessoa@ufpi.edu.br} \thanks{${}^*$ Research partially supported  by CNPq-Brazil}}
{\author{Stefano Pigola ${}^\dagger$ }
\address{Dipartimento di Scienza e Alta Tecnologia\\Universit\`a  dell'Insubria - Como\\
via Valleggio 11, 22100 Como-Italy} \email{stefano.pigola@uninsubria.it} \thanks{${}^\dagger$ Research partially supported  by  the Italian GNAMPA}}
{\author{Alberto G. Setti ${}^\dagger$}
\address{Dipartimento di Scienza e Alta Tecnologia\\Universit\`a  dell'Insubria - Como\\
via Valleggio 11, 22100 Como-Italy} \email{alberto.setti@uninsubria.it}
\date{\today}

\begin{abstract}
We shed a new light on the $L^{1}$-Liouville property for positive, superharmonic functions by providing many evidences that its validity relies on geometric conditions localized on large enough portions of the space. We also present examples in any dimension showing that the $L^{1}$-Liouville property is strictly weaker than the stochastic completeness of the manifold. The main tool in our investigations is represented by the potential theory of a manifold with boundary subject to Dirichlet boundary conditions. The paper incorporates, under a unifying viewpoint, some old and new aspects of the theory, with a special emphasis on global maximum principles and on the role of the Dirichlet Green's kernel.
\end{abstract}

\maketitle

\section{Introduction}
It is well known that most potential theoretic properties of a complete Riemannian manifold only depend on its geometry at infinity, and more precisely, on the properties of its ends (namely the unbounded connected components) with respect to a given compact set. By way of example, a manifold is parabolic, stochastically complete or Feller, if and only if so are all of its ends. See, e.g., \cite{Grigoryan_Analytic, Li} concerning parabolicity, \cite{Bessa-Bar, BPS-Pota} for the stochastic completeness, and \cite{PS-JFA} for the Feller property. It follows that natural geometric conditions implying the validity of each of these properties on a complete manifold, such as volume growth conditions or curvature bounds, are required only outside a large enough compact set.\smallskip

A somewhat less traditional potential theoretic property that, so far, has not been completely emerged, is represented by the Liouville property for $L^{1}$ superharmonic functions. More precisely, we say that the $L^1$-Liouville property holds on  a Riemannian manifold $(M, g)$, or, in short, $M$ is $L^1$-Liouville, if every nonnegative, $L^1$ superharmonic function on $M$ is constant (and therefore vanishes whenever $M$ has infinite volume).
According to works by A. Grigor'yan, \cite{Grigoryan_stochastic_harmonic, Grigoryan_Analytic}, this is equivalent to the fact that the positive, minimal Green's kernel of the Laplace operator is not integrable. Roughly speaking, in stochastic terms this means that we have an infinite global expectation of the Brownian trajectories to explode in infinite time. Thus, parabolic and, more generally, stochastically complete  manifolds are $L^{1}$-Liouville in a trivial way.\smallskip

Some new insights into this property  were recently provided in \cite{BPS-Pota} where the authors construct two examples of  stochastically incomplete $L^1$-Liouville manifolds. Since stochastic completeness is equivalent to  the $L^1$-Liouville property  in models, the examples had to display a strong anisotropy with respect to any fixed origin. In both cases the idea is to start with a stochastically incomplete manifold, and to change the metric in a portion of the manifold so as to ensure that  the $L^1$-Liouville property holds.
This is achieved by using the invariance property of the Green's kernel under conformal changes of the metric, and for this reason the examples produced were two dimensional. The case of generic dimensions remained open. It should be pointed out that in one of the examples the manifold is obtained by gluing together two model manifolds, and therefore has two ends, while in the other the manifold has only one end, and the construction is carried out changing the metric in a portion of the same end.\smallskip

Once it is realised that, in general, the $L^{1}$-Liouville property is weaker than the stochastic completeness of the manifold, one is naturally led to investigate which geometric conditions imply its validity. In fact, the understanding of this property, in connection with the geometry of the space, has not been reached yet. In this direction, a heuristic argument indicates a possible strategy to attack the problem. Indeed, from the Brownian motion viewpoint, the stochastic completeness of $(M,g)$ means that the explosion time $\xi$ of the Brownian trajectories $X_{t}$ issuing from a point $x \in M$ is almost surely infinite. On the other hand, as we have mentioned before, the $L^{1}$-Liouville property could be translated in stochastic terms by saying that the mean of $\xi$ is infinite. Therefore, if sufficiently many trajectories $X_{t}$ escape to infinity in finite time we should have that $M$ is $L^{1}$-Liouville although $\xi$ is not a.s. infinite (stochastic incompleteness). Thus, we expect that the $L^{1}$-Liouville property holds on $M$ if there are sufficiently many directions where $X_{t}$ can explode in finite time. According to this picture, it is reasonable to guess that geometric conditions implying the $L^{1}$-Liouville property are similar to those forcing the stochastic completeness of the manifold, but they can be localized on sufficiently large portions of the space.
\smallskip

The main purpose of the present paper is to pursue this kind of investigation and to obtain information on the validity of the $L^1$-Liouville property of a manifold in terms of geometric properties of suitably large domains thereof. Specifically, we first point out that, as for parabolicity and stochastic completeness, the fact that a complete manifold is $L^{1}$-Liouville depends on its geometry at infinity (Proposition \ref{prop-L1-ends}, Theorem \ref{asymptotic_invariance_L1}). Next, we consider the case where a nonnegative Ricci curvature is concentrated into a geometric half space (Theorem \ref{thm_L1_nonnegative_curv}) and the case where a pinched negative curvature is localised in the complement of a (hyperbolic) half space (Theorem \ref {th-negativecurv}). Finally, manifolds containing warped product cones over a large compact hypersurface with boundary are proved to be $L^{1}$-Liouville by exhibiting a sharp measure in terms of the bottom of the spectrum of the hypersurface (Theorem \ref{th-cone}). \smallskip

A natural way to proceed is to regard such domains, or (possibly unbounded) parts of them, as independent manifolds with boundary, and investigate their potential theoretic properties in connection with the $L^1$-Liouville property of the whole manifold. Among the different possible boundary conditions, the Dirichlet one seems to be most relevant to the problem. \smallskip

Thus we consider the Dirichlet $L^1$-Liouville property of a manifold with boundary and its relationships with the integrability of its Dirichlet Green's kernel (Theorem~\ref{thm_D_L_1}). The main point is that if a manifold with boundary $(M,g)$ sits isometrically inside a manifold $(N,h)$ of the same dimension and $M$ is Dirichlet $L^1$-Liouville then $N$ is $L^1$-Liouville (Corollary~\ref{thm_comparison_dirichlet}). For this reason, a considerable effort is made to study the Dirichlet Green's kernel of a manifold with boundary, which we prove always exists (Theorem~\ref{th-DGreen}) and coincides with the Green's kernel of the interior set. Despite the fact that  the Dirichlet parabolicity of a manifold with boundary turns out to be unrelated to the existence of its Green's kernel, it plays an important role in our study. In \cite{IPS_Crelle}, but see also the Appendix of the present paper, it is observed that the Dirichlet parabolicity is implied by the corresponding property under Neumann conditions, thus showing for the first time a connection between Dirichlet parabolicity and the subquadratic volume growth of intrinsic balls. Although the geometry underlying the notion of Dirichlet parabolicity is still elusive, we feel that our investigations will be relevant in other areas such as in minimal surface theory. The potential theory of manifolds with boundary, especially  related to the notion of Dirichlet parabolicity, has attracted a considerable interest by the minimal surfaces community. The reader may consult the book \cite{MP-book} for an account of recent results and conjectures. \smallskip

We note that since the existence of a minimal positive Green's kernel amounts to the sub-criticality of the positive Laplace-Beltrami operator, there is an intriguing and stimulating relationship between Dirichlet parabolicity on the one hand and criticality theory and the behaviour of the minimal  Martin's kernels on the other. In this framework of investigation one obtains among other things an alternative proof of the existence of the Dirichlet
Green's kernel of a manifold with boundary.

Last, but not least, we mention that the tools developed to study the interplay between geometry and the $L^{1}$-Liouville property allow us to produce in a very direct way some examples in any dimension of $L^1$-Liouville manifolds that are stochastically incomplete (Section~\ref{StochVsL1}), thus completing the picture described in \cite{BPS-Pota} in dimension $2$.\smallskip

The paper is organized as follows:

\begin{itemize}
\item [-] In Section \ref{Sect_PotDir} we introduce some potential theory with Dirichlet boundary conditions on Riemannian manifolds with boundary. We both record, with complete proofs, known facts concerning the Dirichlet parabolicity of a manifold and we provide a new characterization in terms of an Ahlfors property for bounded subharmonic functions. This should be compared with the analogous characterization in the Neumann setting where subharmonic functions are only bounded from above. Moreover, we discuss in detail the existence of the minimal positive fundamental solution of the Laplace equation with Dirichlet boundary conditions and we show how this latter is (un)related to Dirichlet parabolicity. In the  final part of the section we describe the relationships between Dirichlet parabolicity and criticality theory alluded to above.

\item [-] Section \ref {section-DL1Liouville} is devoted to the notion of Dirichlet $L^{1}$-Liouville property on a manifold with nonempty boundary and its characterization in terms of integrability properties of the Dirichlet Green's kernel. This will be equivalent to the fact that the global Dirichlet mean exit time (also called the torsional function of the manifold) is a genuine function. It is also proved, via comparison arguments, that a Riemannian manifold without boundary is $L^{1}$-Liouville provided it contains a family of domains with smooth boundary whose global mean exit times diverge. Finally, it is obtained that the $L^{1}$-Liouville property, for manifolds with or without boundary, depends only on the geometry at infinity.

\item [-] In Section \ref{StochVsL1}, using the tools developed in the previous section, we give examples in any dimension of stochastically incomplete manifolds possessing the $L^{1}$-Liouville property.

\item [-] In Section \ref{section-localized} we move the first steps in the direction of proving that the $L^{1}$-Liouville property depends on geometric conditions localized on sufficiently large portion of the space. More precisely, in terms of curvature conditions, we show that the validity of the $L^{1}$-Liouville property is implied by a nonnegative Ricci curvature condition on a half space or a pinched negative curvature in the complement of a hyperbolic half space. We also give a sharp result concerning the Dirichlet $L^{1}$-Liouville property on a warped product cone and deduce that a manifold is $L^{1}$-Liouville once it contains a sufficiently large cone.

\item [-] A final Appendix discusses the equivalence of the Neumann and the Dirichlet parabolicity on manifolds with compact boundary (ends). In the route, we take the occasion to provide a complete proof of a version of the \textit{boundary point Lemma} and of the strong maximum principle for weak subsolutions of the Laplace equation on a manifold with boundary. These, in particular, justify completely some of the constructions made in \cite{IPS_Crelle}.
\end{itemize}

\section{Some potential theory with Dirichlet boundary conditions}\label{Sect_PotDir}

Let $(M, g)$  be a  Riemannian manifold with nonempty boundary $\partial M$ (defined as the set of points which have a smooth chart $(U,\phi)$ such that $\phi$ is a homeomorphism onto a half ball of $\mathbb{R}^m$ and the coefficients of the metric tensor in these coordinates extend smoothly to an open subset of $\mathbb{R}^{m}$). We denote by $\inte\, M = M\backslash \partial M$ the interior of $M$ (as a manifold with boundary). Following the notation of \cite{Grigoryan_laplace_eq} and \cite{IPS_Crelle}, for any open set $\Omega$ in $M$ we set $\inte\, \Omega = \Omega \cap \inte\, M$, $\partial\Omega= \partial(\inte \,\Omega)$ and let $\partial_0 \Omega = \partial \Omega \cap \inte\, M$ be the Dirichlet boundary of $\Omega$, and  $\partial_1 \Omega =  \bar{\Omega} \cap \partial M$ be the Neumann boundary of $\Omega$, so that $\partial\Omega= \partial_0\Omega\cup \partial_1\Omega$.

In this section we will describe some aspects of potential theory concerning the Dirichlet parabolicity and the Dirichlet Green's kernel of manifolds with boundary. These results will be helpful in the study of Dirichlet $L^1$-Liouville property which will be introduced in the next section.

\subsection{Dirichlet parabolicity}

As in the case of manifolds without boundary, Dirichlet parabolicity is better defined in terms of a constraint in the set of harmonic function (see \cite{Perez-Lopez_MRSI}, \cite{IPS_Crelle}).
\begin{definition}
We say that a manifold $M$ with nonempty boundary $\partial M$ is Dirichlet  parabolic, or in short, $\mathcal{D}$-parabolic if every bounded function $u\in C^\infty(\inte\, M)\cap C^0(M)$ satisfying
\begin{equation}
\label{D-parabolicity-def}
\left\lbrace
\begin{array}{rl}
\Delta u = 0 &\text{in} \ \  \mathrm{int}\, M\\
u=0 &\text{on} \ \ \partial M,
\end{array}\right.
\end{equation}
vanishes identically.
\end{definition}

\begin{remark}
 {\rm
Actually, the previous definition does not require that the manifold at hand has a smooth boundary. By way of example, the notion of $\D$-parabolicity could be extended to any closed domain $\bar \Omega \subset M$ inside the smooth manifold with boundary $M$. Occasionally, where there is not danger of confusion, we shall use this slight abuse of notation.
 }
\end{remark}

In some cases it is useful to compare the Dirichlet parabolicity with the more usual notion of Neumann parabolicity. Recall that $M$ is said to be (Neumann) $\mathcal{N}$-parabolic, or simply parabolic, if any bounded from above, weak solution of
\begin{equation}
\label{Neumann_parabolicity}
\left\lbrace\begin{array}{rl}
\Delta u \geq 0 &\text{in} \ \ \intM\\
\frac{\partial u}{\partial \nu}\leq 0 &\text{on} \ \ \partial M,
\end{array}\right.
\end{equation}
is constant. Here $\nu$ is the outward unit normal to $\partial M$. By a weak solution of
\eqref{Neumann_parabolicity}, we mean a function $u\in C^0(M)\cap W^{1,2}_{loc}(\intM) $ such that
\begin{equation}
\label{Neumann_parabolicity2}
-\int_{\intM} \langle \nabla u,\nabla \rho\rangle \geq 0, \qquad \forall\, 0\leq \rho\in C^\infty_c(M).
\end{equation}

Note that by a density argument (see, e.g., \cite{IPS_Crelle}, or \cite[Corollary 7]{PV_Extension}) one may
consider test functions $\rho\in W^{1,2}_{loc}(M)$.

In \cite{IPS_Crelle} the authors characterized the $\N$-parabolicity of manifolds with nonempty boundary in terms of an Ahlfors maximum principle, namely, given a smooth domain $\Omega \subset M$ with $\partial_{0} \Omega \not= \emptyset$, any bounded from above function $u \in C^{0}(\bar{\Omega})\cap W^{1,2}_{loc}(\inte\,\Omega) $ solution of
$$\left\lbrace
\begin{array}{rl}
\Delta u \geq 0 &  \text{in} \ \ \inte\, \Omega \\[0.2cm]
\frac{\partial u}{\partial \nu} \leq 0 & \text{in} \ \ \partial_1 \Omega,\\[0.2cm]
\end{array}\right.
$$
satisfies
$$ \sup_{\Omega} u = \sup_{\partial_{0} \Omega} u.$$

Observing that when $\Omega = M$ the above Neumann boundary condition plays no role, they obtained the following
\begin{corollary}[Theorem 0.10 in \cite{IPS_Crelle}]
\label{cor_IPS}
Let $M$ be a $\N$-parabolic manifold with nonempty boundary $\partial M$. If $u \in C^{0}(M)\cap W^{1,2}_{loc}(\intM) $ satisfies
$$\left\lbrace
\begin{array}{l}
\Delta u \geq 0 \ \ \text{in} \ \ \inte\, M \\[0.2cm]
\sup_{M} u < + \infty,
\end{array}\right.
$$
then
$$ \sup_{M} u = \sup_{\partial M} u.$$
\end{corollary}

It follows directly from Corollary \ref{cor_IPS} that every $\N$-parabolic manifold is $\mathcal{D}$-parabolic. However, the converse does not hold in general.
\begin{example}\label{ex_D_parabolicity}
{\rm
The closed half space $\bar \rr^{m+1}_{+} = \{(x,y) \in \rr^{m} \times \rr : y\geq0 \}$ of $\rr^{m+1}$, $m\geq 2$ is $\D$-parabolic. Indeed, let $u$ be a bounded harmonic function on $\rr^{m+1}_{+}$ which vanishes on $\partial \rr^{m+1}_{+}$. Then its odd reflection is bounded and harmonic on $\rr^{m+1}$ and therefore identically equal to zero by Liouville's theorem (see, e.g., \cite[Example 4.6]{Grigoryan_Analytic}). An alternative proof, which may be amenable to extensions to more general geometric settings, depends on the fact that, using  the Poisson kernel of the half space it can be proved that, for every bounded function $u \in C^{0}(\bar\rr^{m+1}_{+})$ which is harmonic in $\rr^{m+1}_{+}$ one has the estimate (\cite[Theorems 7.3, 7.5]{ABR})
\[
\| u \|_{L^{\infty}(\rr^{m+1}_{+})} \leq \| u \|_{L^{\infty}(\partial \rr^{m+1}_{+})},
\]
which, again, yields the $\D$-parablicity of $\bar \rr^{m+1}_{+}$. On the other hand, it is not Neumann parabolic, since it possesses a Neumann Green's kernel which is given explicitly in terms of the entire Green's kernel of $\rr^{m+1}$ by
\[
{}^{\mathcal{N}}\!G(x,y) = G(x,y)+G(x,y') \quad \text{where} \quad y'=(y_1,\dots,y_m, -y_{m+1}).
\]
}
\end{example}

Nevertheless, Dirichlet and Neumann parabolicity are equivalent in the case of compact boundary, and in particular for ends of an ambient manifold, viewed as manifolds with boundary. This is the content of Appendix \ref{appendix}. The hierarchy between $\mathcal{N}$-parabolicity and $\mathcal{D}$-parabolicity yields that all geometric conditions that imply the former are sufficient conditions for the validity of the latter. For instance, we have the following proposition due to A. Grigor'yan \cite{Grigoryan_laplace_eq} (see also \cite[Theorem 4.6, Remark 4.8]{IPS_Crelle} for a PDE proof in the $C^1$ case).

Recall that the Riemannian manifold $(M,g)$ with boundary $\partial M \not= \emptyset$ is complete if the metric space $(M,d)$ is complete, where $d(x,y)$ denotes the intrinsic distance defined as the infimum of the lengths of the piecewise $C^{1}$-paths connecting $x$ and $y$. In a complete manifold $(M,g)$ with boundary $\partial M \not=\emptyset$ the metric balls $B_{r}(o)=\{x\in M : d(x,o) < r\}$ and the metric spheres $\partial B_{r}(o)=\{x\in M : d(x,o) = r\}$ are relatively compact sets.

\begin{proposition}
\label{vol-growth-D-parabolicity}
Let $(M,g)$ be a complete Riemannian manifold with nonempty boundary $\partial M$, and assume that either
\begin{equation}
\tag{1}
\frac 1{\mathrm{vol}(\partial B_o(r))}\not\in L^1(+\infty)
\end{equation}
or
\begin{equation}
\tag{2}
\frac r{\mathrm{vol}(B_o(r))}\not\in L^1(+\infty),
\end{equation}
holds for some origin $o \in \inte\, M$. Then $(M,g)$ is $\mathcal{N}$- and therefore $\mathcal{D}$-parabolic.
\end{proposition}

\begin{remark}\rm
As a consequence of the co-area formula the volume growth condition (2) implies the area growth
condition (1).
\end{remark}
In view of applications it is worth to remark some important equivalent forms of $\D$-parabolicity. We begin with a useful characterization described in \cite{Perez-Lopez_MRSI} and that in a germinal form can be traced back to \cite{Grigoryan_Analytic} (see Remark \ref{nonmassiveness_grigor} below). If  $\{\Omega_k\}$ is an exhaustion of $M$ by an increasing sequence of relatively compact open sets $\Omega_k \Subset \Omega_{k+1}$ with $\partial_0\Omega_k\ne\emptyset$ transversal to $\partial M$, we denote by  $v_k$ the solution of the Dirichlet problem
\begin{equation}\label{D-parab-exhaustion}
\left\lbrace
\begin{array}{rl}
\Delta v_k = 0 & \text{ in}\ \  \inte\, \Omega_{k}\\
v_k=0& \text{ on} \ \ \partial_1 \Omega_k \\
v_k=1& \text{ on} \ \ \partial_0 \Omega_k.
\end{array}\right.
\end{equation}
Note that the solution $v_k$ which may be obtained by a standard application of the Perron procedure, satisfies $0\leq v_k\leq 1$ and is smooth at all points except at ${\partial_1 \Omega_k} \cap \overline{\partial_0 \Omega_k}$. Since $v_{k+1}$ is continuous on $\bar \Omega_k$, by  the usual weak maximum principle, the sequence $\{v_k\}$ is decreasing.

\begin{proposition}[\cite{Grigoryan_Analytic, Perez-Lopez_MRSI}]
\label{DirPar_equiv}
Let $M$ be a Riemannian manifold with nonempty boundary $\partial M$. The following are equivalent:
\begin{enumerate}
\item[(i)] $M$ is $\mathcal{D}$-parabolic;
\item[(ii)] For every increasing exhaustion $\Omega_k$ with the properties described above, the solution $v_k$ of the Dirichlet problem \eqref{D-parab-exhaustion} satisfies $v_k\to 0$ on $M$ as $k\to +\infty$;
\item[(iii)] There exists an increasing exhaustion $\Omega_k$ with the properties described above for which $v_k\to 0 $ as $k\to +\infty$.
\end{enumerate}
\end{proposition}
\begin{proof}
Assume first that $M$ is $\mathcal{D}$-parabolic and let $\Omega_k$ be any increasing exhaustion of $M$ as in the statement and let $v_k$ be the solution of the corresponding Dirichlet problem  \eqref{D-parab-exhaustion}. By monotonicity the sequence $v_k$ converges locally uniformly on $M$ to a bounded solution $v$ of \eqref{D-parabolicity-def}, and therefore $v=0$ by definition of $\D$-parabolicity, showing that $(i)$ implies $(ii)$.

It is clear that (ii) implies (iii). To prove that (iii) implies (i), assume that there exists an exhaustion $\Omega_k$ such that the corresponding sequence $\{v_k\}$ tends to zero as $k\to +\infty$. Let $u\in C^2(\inte\, M) \cap C^0(M)$  be bounded and satisfy $\Delta u=0 $ in $\inte\, M$, $u=0$ on $\partial M$. By scaling we may assume without loss of generality that $|u|\leq 1$ on $M$. Since  $u\leq  v_k$ on $\partial \Omega_k$, by the weak maximum principle $u\leq  v_k$ on $\Omega_k$ and letting $k\to +\infty$, $u\leq 0$. Applying the argument to $-u$ we conclude that $u=0$ and $M$ is $\mathcal{D}$-parabolic.
\end{proof}

\begin{remark}\label{nonmassiveness_grigor}\rm
We observe that when $M$ is a smooth domain of a smooth manifold $N$ without boundary, the characterization given in Proposition \ref{DirPar_equiv} coincides with the notion of nonmassiveness of $M$ presented in \cite{Grigoryan_Analytic}, that is, the property that $M$ does not support a nonnegative, bounded subharmonic function such that $v = 0$ on $N \backslash \inte\,M$ and $\sup_{M}v > 0$. In fact, \cite[Proposition 4.3]{Grigoryan_Analytic} shows that the harmonic measure of the complementary set of $M$ is the limit of the solutions to problems \eqref{D-parab-exhaustion}, and it is equal to $0$ if and only if $M$ is nonmassive. In the case of manifolds without boundary the existence of a proper massive set turns out to be equivalent to the nonparabolicity of the manifold (\cite[Theorem 5.1]{Grigoryan_Analytic}). As we will see in the next subsection, Lemma \ref{lemma_extension}, every Riemannian manifold with nonempty boundary can be considered as a domain in a Riemannian manifold without boundary, however this extension is nonparabolic.
\end{remark}

Proposition \ref{DirPar_equiv} implies a version of the Ahlfors maximum principle.

\begin{proposition}
\label{Dirichlet_Ahlfors}
Let $(M,g)$ be a Riemannian manifold with nonempty boundary $\partial M$. The following are equivalent:
\begin{itemize}
\item[(i)] $M$ is $\mathcal{D}$-parabolic;
\item[(ii)] for every bounded function $u\in C^{\infty}(\mathrm{int}\,M)\cap C^0(M)$ such that $\Delta u=0$ in $\inte\, M$
we have
\begin{equation*}
\sup_M u = \sup_{\partial M} u \quad \text{and}\quad \inf_M u = \inf_{\partial M} u;
\end{equation*}
\item[(iii)] for every domain $\Omega$ contained in $M$ and every bounded function $u\in C^{\infty}(\mathrm{int}\, \Omega) \cap C^0(\bar \Omega)$ satisfying $\Delta u =0$ in $\mathrm{int}\,\Omega$ we have
\begin{equation*}
\sup_\Omega u= \sup_{\partial \Omega} u \quad \text{and}\quad \inf_\Omega u = \inf_{\partial \Omega} u.
\end{equation*}
\end{itemize}
\end{proposition}

\begin{proof}
It is trivial that (iii) implies (ii) which in turn implies (i).

It remains to prove that (i) implies (iii). Suppose  that $M$ is $\mathcal{D}$-parabolic,  let $\Omega$ be a domain in $M$ and  let $u\in C^{\infty}(\mathrm{int}\, \Omega)\cap C^0(\bar \Omega)$ be bounded and satisfy $\Delta u=0$ in $\mathrm{int}\, \Omega$. By scaling we may assume without loss of generality that $|u|\leq 1$ on $M$.

Consider $\tilde u= u-\sup_{\partial \Omega} u$, so that $\tilde u\leq 0$ on $\partial \Omega$ and $\Delta \tilde u= 0$ in $\mathrm{int}\, \Omega$. Let $\{\Omega_k\}$ be an increasing exhaustion of $M$ and $v_k$ be the sequence of functions described in Proposition \ref{DirPar_equiv}. Since $\tilde u\leq v_k$ on $\partial (\Omega\cap\Omega_k)$, by the weak maximum principle we have $\tilde u \leq v_k$ on $\Omega\cap \Omega_k,$ and letting $k\to +\infty$, $\tilde u\leq 0$ on $\Omega$, that is, $u\leq \sup_{\partial \Omega} u$. The
inequality $u \geq \inf_{\partial \Omega} u$ is obtained applying the same argument to $-u$.
\end{proof}

Perhaps surprisingly, the Dirichlet parabolicity implies a stronger version of the Ahlfors' maximum principle, which involves subharmonic functions and elucidates the ultimate difference between Neumann and Dirichlet parabolicity. The latter deals with bounded subharmonic functions, whose supremum is attained on $\partial \Omega = \partial_0\Omega \cup \partial_1 \Omega$ while in the former one considers bounded above subharmonic functions, which attain their supremum on $\partial_0\Omega$.

\begin{proposition}
\label{DirParSubharmonic}
Let $M$ be a manifold with boundary $\partial M \not= \emptyset$. Then the following are equivalent:
\begin{itemize}
\item[(i)] $M$ is $\mathcal D$-parabolic;
\item[(ii)] for every domain $\Omega\subset M$ and every bounded function $u \in C^{0}(\Omega)\cap W^{1,2}_{loc}(\inte\, \Omega)$ satisfying
$\Delta u \geq 0$ on  $\mathrm{int}\, \Omega,$ it holds
\[
\sup_{\Omega } u = \sup_{\partial \Omega} u;
\]
\item[(iii)]
For every bounded function $u \in C^{0}(M)\cap W^{1,2}_{loc}(\inte\,M)$ satisfying $\Delta u\geq 0$ on $\intM$ it holds
\[
\sup_M u= \sup_{\partial M} u.
\]
\end{itemize}
\end{proposition}
\begin{proof}
(i) $\Leftrightarrow$ (iii). Since, by Proposition~\ref{Dirichlet_Ahlfors}, (iii) implies (i) we have only to consider the reverse implication. Suppose by contradiction that $u$ is a bounded function satisfying $\Delta u\geq 0$ on $\inte\,M$ and $\sup_M u> \sup_{\partial M} u = \mu$.
Define
\[
\tilde u (x) = \frac{\max\{u(x) - \mu,0\}}{\sup_{M} |u| +|\mu|}
\]
and observe that $0 \leq \tilde u \in C^{0}(M) \cap W^{1,2}_{loc}(\inte\,M)\cap L^{\infty}(M)$  satisfies
\[
\Delta \tilde u\geq 0 \text{ on }\mathrm{int}\, M, \quad \tilde u= 0 \text{ on }\partial M,\quad  0 < \sup_{M} \tilde u \leq 1.
\]
Starting from $\tilde u$ it is now standard to construct a harmonic function $v \in C^{0}(M) \cap C^{\infty}(\inte\,M) \cap L^{\infty}(M)$ such that
\[
v \geq \tilde u \text{ on }M,\quad  v =0 \text{ on }\partial M.
\]
This follows  either via the exhaustion procedure in \cite[p. 157]{Grigoryan_Analytic} or by applying to $- \tilde u$ the reduction technique described in \cite[p. 132 ff, Theorem 4.3.2]{H}. In the latter case, the vanishing of $v$ on $\partial M$ is proved by constructing for every point of $\partial M$ a global barrier  larger than $\tilde u$. Since
\[
\sup_{M} v  > 0 = \sup_{\partial M} v
\]
we get the desired contradiction.\\
(ii) $\Leftrightarrow$ (iii). It is clear that (ii) implies (iii). On the other hand, suppose that (ii) is not satisfied. Then, there exist a domain $\Omega$ in $M$ and a function $u\in C^{0}(\Omega) \cap W^{1,2}_{loc}(\inte\,\Omega) \cap L^{\infty}(\Omega)$ satisfying $\Delta u\geq 0$ in $\mathrm{int}\, \Omega$ and $\sup_\Omega u> \sup_{\partial \Omega} u = \mu$.  If $0 < \epsilon \ll 1$ is small enough, we can construct a function $u_{\epsilon} \in C^{0}(M) \cap W^{1,2}_{loc}(\inte\,M) \cap L^{\infty}(M)$ by setting
\[
 u_{\epsilon} (x) =
\begin{cases}
 \max\{u(x) - \mu - \epsilon,0\} & \text{in} \ \ \Omega \\
 0 & \text{in} \ \ M \backslash \Omega,
\end{cases}
\]
in such a way that the following conditions are satisfied:
\[
\Delta  u_{\epsilon} \geq 0 \text{ in } \inte\,M, \quad u_{\epsilon} =0 \text{ on } \partial M,\quad  \sup_{M}  u_{\epsilon} >0.
\]
This shows that (iii) is not satisfied.
\end{proof}

As a consequence of the above propositions, we have
\begin{corollary}
\label{D-parabolicity of subdomains} Let $M$ be a manifold with boundary and let $\Omega$  be a smooth domain in $M$. If $M$ is $\mathcal{D}$-parabolic then $M\backslash \Omega$ is $\mathcal{D}$-parabolic\footnote{in the extended sense.}. If $\Omega$ is relatively compact, and $M\backslash \Omega$ is $\mathcal{D}$-parabolic, then $M$ is $\mathcal{D}$-parabolic.
\end{corollary}
\begin{proof}
The first statement follows immediately from the previous proposition. Suppose now that $\Omega$ is relatively compact and that $M\backslash \Omega$ is $\mathcal{D}$-parabolic. Let $u\in C^{\infty}(\inte\, M)\cap C^0(M)$ be bounded and satisfy $\Delta u=0$ in $\inte\,M $ and $u=0$ on $\partial M$. Since $M\backslash \Omega$ is $\mathcal{D}$-parabolic,
\[
\sup_{M\backslash \Omega} u = \sup_{\partial (M\backslash \Omega)} u = \sup_{\partial_0 \Omega\cup (\partial M\backslash \partial_1 \Omega)} u.
\]
Since $\Omega$ is relatively compact,  $\sup_\Omega u= \sup_{\partial \Omega} u$ so that, if $\sup_{\partial_0\Omega} u= \sup_{M\backslash \Omega} u $, then $u$ attains a maximum on $\inte\, M$ and $u=0$ from the strong maximum principle. Otherwise,
\[
\sup_{\partial_0 \Omega} u< \sup_{\partial M}  u=0,
\]
and therefore $u\leq 0$.  Arguing in a similar manner, $u\geq 0$, whence $u=0$ as required to show that $M$ is $\mathcal{D}$-parabolic.
\end{proof}

\begin{remark}\rm
In the above corollary, applying Lemma \ref{lemma_extension} below, we can think  of $M$ as a domain in a Riemannian manifold $N$ without boundary, and looking at the $\D$-parabolicity as  nonmassiveness of $M$, we recover  \cite[Proposition 4.2]{Grigoryan_Analytic}.
\end{remark}

Corollary~\ref{D-parabolicity of subdomains} in  turns implies the invariance of $\mathcal{D}$-parabolicity under compact perturbations.
\begin{corollary}
\label{D-parabolicity-invarince} Let $M_1$ and $M_2$ be Riemannian manifolds and suppose that there exist relatively compact sets $ \Omega_1\subset M_1$ and $\Omega_2 \subset M_2$ such that $M_1\backslash \Omega_1$ is isometric to $M_2\backslash \Omega_2$. Then $M_1$ is $\mathcal{D}$-parabolic if and only if so is $M_2$.
\end{corollary}

As a last sufficient condition for  $\D$-parabolicity we have the following  version of the Khas'minskii test.

\begin{lemma}[Khas'minskii test]\label{Khas-D-parabolicity}
Let $M$ be a manifold with boundary $\partial M$. If there exist a compact set $K\subset M$ and a function $0\leq \phi\in C^0(M\backslash \text{\r{K}})\cap W^{1,2}_{loc}(\inte\,M\backslash K)$ such that $\phi(x)\to +\infty $ as $x$ diverges, and
\[
-\int_{\inte\,M\backslash K} \langle \nabla \phi,\nabla \rho\rangle \leq 0\quad \forall \,\,0\leq \rho\in C^0(M\backslash \text{\r{K}})\cap W^{1,2}_{loc}(\inte\,M\backslash K),
\]
then $M$ is $\mathcal{D}$-parabolic.
\end{lemma}
\begin{proof}
Let $u\in C^0(M)\cap C^{\infty}(\intM)$ be bounded  and satisfy $\Delta u =0$ in $\inte\,M$ and $u=0$ on $\partial M$. We prove that $u\leq 0$ on $M$. A similar argument shows that $u\geq 0$ on $M$, whence $u=0.$
Suppose by contradiction that $\sup_M u>0$. Since $u$ cannot attain its supremum in  $\inte\,M$, by the strong maximum principle, and $u=0$ on $\partial M$, $u$ cannot attain its supremum in $M$ and therefore
$\sup_K u <\sup_M u$. Let $\gamma>0$ be such that $\sup_K u < \gamma <\sup_M u$,  pick $x_o \in M\backslash K$ such that $u(x_o)>\gamma$ and let $v=u-\gamma -\epsilon \phi$, where $\epsilon>0$ is small enough that $v(x_o)>0$. Finally, let $\Omega=\{x\,:\, v(x)>0\}$. Then $x_o\in \Omega$, $\bar{\Omega}\cap (K\cup \partial M) =\emptyset$ and since $\phi(x)\to +\infty$ as $x$ diverges,  $\Omega$ is bounded. Moreover
$\Delta v\geq 0$ weakly in $\Omega$ and $v\leq 0$ on $\partial \Omega$, so that, by comparison $v\leq 0$ on $\Omega$, contradiction.
\end{proof}

\subsection{Dirichlet Green's kernel}

We now briefly describe the construction of the Dirichlet Green's kernel ${}^{\mathcal D}\!G$ of a manifold $M$ with boundary $\partial M$, namely,  the minimal positive solution of
\begin{equation}\label{DiricheltGreenFcn}
\Delta_x {}^{\mathcal D}\!G (x, y)= -\delta_y(x) , \,\, \forall \,\, x, y\in \inte\,M,\quad
{}^{\mathcal D}\!G (x, y) =0 \,\text{ if }\, x\text{ or } y \in \partial M.
\end{equation}
As in the case of manifolds without boundary, the Green's kernel ${}^{\mathcal D}\!G (x, y)$ can be defined by an exhaustion procedure. One considers an increasing exhaustion of $M$ by means of relatively compact sets $\Omega_k$ with smooth Dirichlet boundary $\overline{\partial_0\Omega_k}$ intersecting $\partial M$ transversally. Then the sequence $\{{}^{\mathcal D}\!G^{\Omega_k}\}$ of the
Dirichlet Green's kernels of $\Omega_k$ is increasing and by the local Harnack inequality, either it diverges everywhere on $\inte\,M$, or it converges locally uniformly to a smooth function
off the diagonal of $\inte\,M$ satisfying $\Delta_x {}^{\mathcal D}\!G(x,y)= -\delta_y(x)$.
Moreover, if $y\in\inte\,M$ and $r>0$ is such that $\bar B_r(y) \Subset \inte\,M$, then for every $k$ such that  $\bar B_r(y) \Subset \inte\, \Omega_k$ and every $x\in \Omega_k\backslash B_r(y)$ we have
\[
{}^{\mathcal D}\!G^{\Omega_k} (x,y) \leq \sup_{\partial B_r(y)}{}^{\mathcal D}\!G^{\Omega_k}(\cdot, y)
\leq \sup_{\partial B_r(y)} {}^{\mathcal D}\!G(\cdot, y).
\]
By Schauder's boundary estimates, it follows that ${}^{\mathcal D}\!G^{\Omega_k}(\cdot, y)$ converges locally uniformly in $M$ and therefore ${}^{\mathcal D}\!G(\cdot, y)$ is continuous up to the boundary and vanishes  on $\partial M$.

The Neumann Green's kernel ${}^{\mathcal N}\!G$, which is required to satisfy the Neumann boundary condition on $\partial M$ is obtained using a similar limiting procedure, and, since, by comparison,
\[
{}^{\mathcal D}\!G^{\Omega_k}\leq {}^{\mathcal N}\!G^{\Omega_k}, \quad \forall \,\,k,
\]
we have
\[
{}^{\mathcal D}\!G\leq {}^{\mathcal N}\!G.
\]
We are going to show that every manifold with boundary admits a Dirichlet Green's kernel. However, unlike in the Neumann setting, the existence of a Dirichlet Green's kernel is unrelated to the  Dirichlet parabolicity of  the underlying manifold.

\begin{example}[Euclidean half spaces]\label{ex_euclid_half_space_2}
{\rm
We revisit Example \ref{ex_D_parabolicity}. Let
\[
\bar\rr^{m+1}_{+} = \{(x,y) \in \rr^{m} \times \rr : y\geq 0 \}
\]
be the closed half space of $\rr^{m+1}$, $m\geq 1$. If $m=1$, we already know from volume growth considerations, see Proposition~\ref{vol-growth-D-parabolicity}, that
 $\bar \rr^{2}_{+}$ is $\N$-parabolic, hence $\D$-parabolic. On the other hand, since $\rr^{2}_{+}$ is conformally equivalent to the unit disc $D(0;1) \subset \rr^{2}$ via a conformal map that sends $\partial \rr^{2}_{+}$ onto $\partial D(0;1) \backslash \{e^{i \pi/2}\}$, we can  transplant to $\bar \rr^{2}_{+}$ the Dirichlet Green's kernel $u(x,y)=-\log |(x,y)|$ of $\bar D(0;1)$ with pole at the origin.\smallskip

If  $m\geq 2$, harmonicity is no longer a conformal property and the volume growth of $\bar \rr^{m+1}_{+}$ is  much too fast to be related to parabolicity. However, as we had seen in Example \ref{ex_D_parabolicity} the above conclusions can be extended even to $m \geq 2$ by means of different arguments. On the other hand, $\bar \rr^{m+1}_{+}$ possesses a Dirichlet Green's kernel with pole $o = (x_{o},y_{o})\in \rr^{m+1}_{+}$. It is given explicitly by
\[
\G^{\bar\rr^{m+1}_+}(p,o) = C \left\{ \frac{1}{|p - o|^{m-1}} - \frac{1}{|p-o'|^{m-1}} \right\}
\]
where $C=C(m) >0$ is a dimensional constant and $o' = (x_{o},-y_{o})$.
}
\end{example}

\begin{theorem}\label{th-DGreen}
Let $(M,g)$ be a  Riemannian manifold with nonempty  boundary $\partial M$. Then, for every $o \in \inte\, M$, there exists the Dirichlet Green's kernel $\G(x,o)$ of $M$ with pole at $o$.
\end{theorem}
The proof we presently describe relies on the next lemma of independent interest. It states that every manifold with boundary has a nonparabolic Riemannian extension. In the following section we will provide an alternative proof based on criticality theory.

\begin{lemma}\label{lemma_extension}
Let $(M,g)$ be an $m$-dimensional Riemannian manifold with boundary $\partial M \not = \emptyset$. Then, there exists an $m$-dimensional Riemannian manifold $(N,h)$ with $\partial N = \emptyset$ such that:
\begin{itemize}
 \item [(a)] $(M,g)$ is isometrically embedded into $(N,h)$ as a closed subset.
 \item [(b)] $(N,h)$ is nonparabolic (as a manifold without boundary).
\end{itemize}
\end{lemma}
\begin{proof}
We first extend $(M,g)$ past its boundary so to obtain a new Riemannian manifold $(M',g')$ (possibly incomplete) without boundary that contains $(M,g)$ isometrically. Since $M'$ can be obtained by adding to $M$ a collar inside the diffeomorphic double of $M$, we can assume that $M$ is a closed subset; see e.g. \cite{PV_Extension}. Next, we delete from $M' \backslash M$ a small compact ball $\bar B'$. Let $(N,h)$ be the resulting Riemannian manifold. Since $\bar B'$ has positive local capacity in $M'$, $N$ is nonparabolic; \cite[Theorems 3.5, 4.4] {Tr-Siberian}.
\end{proof}

We shall also need smooth exhaustions of an ambient manifold that restrict to Lipschitz exhaustions of a given smooth open submanifold.

\begin{lemma}\label{lemma-transversal}
Let $M$ be an open submanifold  with smooth boundary $\partial M \not= \emptyset$ inside the differentiable manifold $N$ without boundary. Then, there exists a relatively compact exhaustion $\{\Omega^{N}_{k}\} \nearrow N$ of $N$ such that, for every $k$, $\partial \Omega^{N}_{k}$ is a smooth hypersurface intersecting transversally $\partial M$.
 \end{lemma}
\begin{proof}
Fix a relatively compact exhaustion $\{U^{N}_{k}\} \nearrow N$ with smooth boundary $\partial U^{N}_{k} \not=\emptyset$ and set $U^{N}_{-1} = \emptyset$. Using induction on $k$, we are going to modify each smooth compact hypersurface $\partial U^{N}_{2k}$ inside the open manifold $U^{N}_{2k+1} \backslash \bar U^{N}_{2k-1}$ so to obtain a new hypersurface $\Sigma_{2k}$ such that:
\begin{itemize}
 \item [(c)] $\Sigma_{2k}$ bounds a domain $\Omega^{N}_{2k}$ such that $U^{N}_{2k-1} \Subset \Omega^{N}_{2k} \Subset U^{N}_{2k+1}$;
 \item [(d)] $\Sigma_{2k}$ intersects transversally $\partial M$ (possibly in an empty set).
\end{itemize}
To this end, we consider the inclusion
\[
i_{k} : \bar U^{N}_{2k} \hookrightarrow N
\]
and we apply the Transversality Homotopy Theorem, \cite[p. 70]{GP}, to get a smoothly homotopic map
\[
j_{k} : \bar U^{N}_{2k} \hookrightarrow N
\]
such that $\Sigma_{2k} = j_{k}(\partial U^{N}_{2k})$ is transversal to $\partial M$, i.e., property (d) holds. Since both ``maximal rank'' of a map and ``embededness'' are stable properties, \cite[p. 35]{GP}, we can assume that $j_{k}$ and $j_{k}|_{\partial U^{N}_{2k}}$ are still embeddings. Finally, the homotopy can be obtained by modifying $i_{k}$ only in an arbitrarily small neighborhood of $\partial U^{N}_{2k}$, \cite[p. 72]{GP}. Therefore, if we set $\Omega^{N}_{2k} = j_{k}(U^{N}_{2k})$, we can always guarantee that $\Sigma_{2k}$ satisfies also condition (c).
\end{proof}

We are now ready to give the
\begin{proof}[Proof of Theorem \ref{th-DGreen}]
Let $(N,h)$ be a nonparabolic Riemannian extension of $M$ with Green's kernel $G^{N}$. Fix $o \in \inte\, M$ and recall that both $G^{N}(\cdot,o)$ and $\G^{M}(\cdot,o)$ are obtained as the limit of Dirichlet Green's kernel for a chosen relatively compact exhaustion. According to Lemma \ref{lemma-transversal}, let $\{ \Omega^{N}_{k} \} \nearrow N$ be a smooth, relatively compact exhaustion of $N$ such that $o \in \Omega_{0}$ and each intersection $\partial \Omega_{k}^{N} \cap \partial M$ is transversal. For each $k$, define $\Omega^{M}_{k} = \Omega^{N}_{k} \cap M$. Since $M \subset N$ is a closed domain then each $\Omega^{M}_{k}$ is relatively compact. Moreover, by the transversality intersection condition, $\partial \Omega^{M}_{k}$ is Lipschitz. Thus, $\{ \Omega^{M}_{k}\} \nearrow M$ is a good relatively compact exhaustion of  $M$. Let $\G^{\Omega^{N}_{k}}(\cdot,o)$ and $\G^{\Omega^{M}_{k}}(\cdot,o)$ be the Dirichlet Green's kernels with pole $o$ of the corresponding domains. Then, by comparison
\[
\G^{\Omega^{M}_{k}}(x,o) \leq \G^{\Omega^{N}_{k}} (x,o) \leq G^{N}(x,o) \ \ \text{on} \ \ \Omega^{M}_{k},
\]
and taking the limit as $k \to +\infty$ we conclude
\[
\G^{M}(x,o) \leq G^{N}(x,o) \ \ \text{on} \ \ M.
\]
\end{proof}
Note that the above proof shows that if $M$ is a manifold with boundary sitting inside a manifold $N$
then
\begin{equation}\label{ineq_G_D}
{}^{\mathcal D}\!G^M\leq G^N.
\end{equation}

Observe also that, if we define the Green's kernel $G^{\inte\,M}$ of $\inte\,M$ as the limit $\lim_k {}^{\D}\!G^{\Omega'_k}$ where  $\{\Omega'_k\}$ is an increasing exhaustion of $\inte\,M$ by relatively compact (in $\intM$)  open sets with smooth boundary, the above argument shows that       $$G^{\inte\,M}\leq {}^{\mathcal D}\!G^{M}.$$
This latter inequality together with the fact that ${}^{\mathcal D}\!G^{M}(x,y)\to 0$ as $x\to \bar x \in\partial M$ show that $G^{\inte\,M}$ can be extended to a continuous function off the diagonal of $M$. The usual comparison argument shows that, for every open set $\{\Omega_k\}$ in the exhaustion of $M$ chosen to define ${}^{\mathcal D}\!G^{M}$, one has ${}^{\D}\!G^{\Omega_k}\leq G^{\inte\,M}$ off the diagonal of $\Omega_k$ and therefore
${}^{\mathcal D}\!G^{M} \leq G^{\inte\,M}$.

We summarize the above discussion in the following

\begin{theorem}\label{subcriticality}
Let $M$ be a manifold with boundary $\partial M\ne \emptyset$. Then for every $o\in\inte\, M$ there exists a minimal positive Green's kernel $G^{\inte\,M}(\cdot, o)$  with pole at $o$. Moreover $G^{\inte\,M}(\cdot, o)={}^{\mathcal D}\!G^{M}(\cdot, o)$.
\end{theorem}

\subsection{Dirichlet potential theory from the viewpoint of the Criticality Theory}\label{subsect-Criticality}

The Dirichlet parabolicity of a manifold with boundary, as well as the existence of the minimal positive Green's kernel subject to Dirichlet boundary conditions, can be considered in the framework of the \textit{criticality theory} for Schr\"odinger operators. This theory was initiated by B. Simon \cite{Si},  developed by M. Murata \cite{Mu, Mu1} and Y. Pinchover \cite{Pi, Pi1, Pi2}, and it has been recently extended to the setting of semilinear operators and general boundary conditions, see e.g. \cite{DePi, DeMaPi, PiSa}.  A worth reference in the classical Dirichlet case for linear operators is the book of R. G. Pinsky \cite{Pinsky}.\smallskip

Let $(M,g)$ be an $m$-dimensional Riemannian manifold with smooth boundary $\partial M \not=\emptyset$. As we have already remarked, there is no loss of generality in considering $M$ as a smooth domain inside an equidimentional ambient manifold $(N,h)$.  Using the viewpoint of the criticality theory, we are going to show that:\smallskip

\noindent - $\inte\, M$ has a minimal positive Green's kernel which, in fact, is the restriction to $\inte\, M$ of the Dirichlet Green's kernel of $M$. In particular, $\inte\, M$ has a Martin boundary $\pma M$.\smallskip

\noindent - There is a connection of the $\D$-parabolicity of $M$ and the behavior of the Martin kernels with poles in $\pma M \backslash \partial M$.\smallskip

\noindent - A natural extension of the notion of parabolicity can be investigated  for Schr\"odinger-type operators $P = -\Delta +V$  subject to more general boundary conditions.\smallskip

\noindent We are indebted to the anonymous referee for having pointed out all of these facts to us.

To begin with, we recall that a second order elliptic operator $P$ defined in a domain
$\Omega$ in a Riemannian manifold is subcritical if it admits a positive minimal Green's kernel, supercritical if there are no positive solutions of $Pu=0$ and critical if it is neither sub nor supercritical, that is it does not admit a minimal positive Green's kernel but there exist positive solutions of $Pu=0$.

We give the following
\begin{lemma}\label{lemma-criticality-green}
 The Riemannian manifold without boundary $(\inte\, M,g)$ has a minimal positive Green's kernel $G^{\inte\, M}(x,y)$.
\end{lemma}
\begin{proof}
Let us define the cone
\[
\CSH_{-\Delta}(\inte\, M) =\{ u \in C^{2}(\inte\, M): u>0\text{ and }\Delta u \leq 0\}.
\]
It is enough to prove that
\begin{equation}\label{criticality1}
\dim \CSH_{-\Delta}(\inte\, M) >1.
\end{equation}
Indeed, according to \cite[Chapter 4]{Pinsky} or \cite[Theorem 1.4]{PiSa},
this condition is equivalent to the subcriticality of the (positive definite) Laplace-Beltrami operator $-\Delta$ in $\inte\, M$.

Now, let $x,y$ be two distinct points in $\partial M$, and let $I_{x},I_{y}\subset \partial M$ be disjoint neighborhoods of $x$ and $y$. Consider an exhaustion $\{\Omega_{k}\}$ of $M$ as in Proposition \ref{DirPar_equiv} and, for every $k \geq 1$, let $u_{k,x}$ and $u_{k,y}$ be the solutions of the Dirichlet problems
\[
\begin{array}{cccc}
 \begin{cases}
 \Delta u_{k,x} = 0& \text {in } \Omega_{k}\\
 u_{k,x} =0 &\text{on }\partial \Omega_{k} \backslash I_{x}\\
 u_{k,x} =1 &\text{on }I_{x},
\end{cases}
&
\begin{cases}
 \Delta u_{k,y} = 0& \text {in } \Omega_{k}\\
 u_{k,y} =0 &\text{on }\partial \Omega_{k} \backslash I_{y}\\
 u_{k,y} =1 &\text{on }I_{y}.
\end{cases}
\end{array}
 \]
Then $u_{x}= \lim_{k\to+\infty} u_{k,x}$ and $u_{y} = \lim_{k\to+\infty} u_{k,y}$ are two linearly independent positive harmonic functions on $\inte\, M$, proving the validity of \eqref{criticality1}.
\end{proof}
Using Lemma \ref{lemma-criticality-green} we obtain an alternative proof of Theorems \ref{th-DGreen} and \ref{subcriticality}.

\begin{theorem}
 Let $(M,g)$ be a Riemannian manifold with smooth boundary $\partial M \not=\emptyset$. Then $M$ has a Dirichlet Green's kernel $\G^{M}(x,y)$. Moreover, $\G^{M}(x,y) = G^{\inte\, M}(x,y)$ on $\inte\, M$.
\end{theorem}

\begin{proof}[Alternative proof]
By Lemma \ref{lemma-criticality-green}, $\inte\, M$ supports a minimal positive Green's kernel $G^{\inte\, M}(x,y)$.  This function extends continuously to $\partial M$ with zero boundary values, see e.g. \cite[Chapter 8, Theorem 1.1]{Pinsky}. By comparison along an exhausting sequence $\{\Omega_{k}\}$ of $M$ we deduce that
\[
\G^{\Omega_{k}}(x,y) \leq G^{\inte\, M}(x,y)\, \text{ on } \Omega_{k}
\]
and, therefore,
 $\G^{M}(x,y)= \lim_{k\to+\infty} \G^{\Omega_{k}}(x,y)$ exits and satisfies
\[
\G^{M}(x,y) \leq G^{\inte\, M}(x,y).
\]
On the other hand, using  a comparison argument on an exhausting sequence $\{D_{k}\}$ of $\inte\, M$ we have that
\[
G^{\inte\, M}(x,y)  = \lim_{k\to+\infty}G^{D_{k}}(x,y)\leq \G^{M}(x,y).
\]
This completes the proof.
\end{proof}
Once it is realized that $\inte\, M$ is a \textit{Greenian domain} with respect to the Laplace operator or, equivalently, that $-\Delta$ is subcritical on $\inte\, M$, we can apply the machinery of the Martin compactification to give an interpretation of the $\D$-parabolicity in terms of properties of certain Martin kernels. Indeed, recall from \cite{It}, \cite{Mu1} and \cite[Chapter 8, Theorem 1.4, Corollary 1.6]{Pinsky} that the Martin boundary $\pma M$ of $\inte\, M$ decomposes as
\[
\pma M = \partial M\, \dot\cup\, \pma' M
\]
where:\smallskip

\noindent - $\partial M$ is the part  of $\pma M$ corresponding to sequences $\{x_{k}\} \subset \inte\, M$ that converge to some point of $\partial M$;\smallskip

\noindent - $\pma' M$ is the part of $\pma M$ corresponding to sequences $\{y_{k}\}\subset \inte\, M$ without accumulation points in $\partial M$;\smallskip

\noindent - $\partial M$ embeds into the minimal Martin boundary $\pma^{0}M$;\smallskip

\noindent - $\pma' M \not= \emptyset$ and $\pma' M \cap \pma^{0} M \not=\emptyset$ provided that the cone
\[
\CH^{0}_{-\Delta}(\inte\, M) = \left\lbrace
u \in C^{\infty}(\inte\, M)\cap C^{0}(M) :
\begin{array}{rrl}
\Delta u =0 & \text{in} \ \ \inte\, M;\\
u>0 & \text{in} \ \ \inte\, M;\\
u=0 & \text{on} \ \ \partial M;
\end{array}
\right\rbrace
\]
is nonempty\footnote{Indeed, given $u \in \CH^{0}_{-\Delta}(\inte\, M)$, the corresponding Martin measure $d\mu_{u}(\zeta)$ is supported in $\pma^{0}M \cap \pma' M$.};\smallskip

\noindent - for every $\zeta \in \pma' M \not=\emptyset$ the corresponding Martin kernel $k(x,\zeta)$ with pole at $\zeta$ is a positive harmonic function on $\inte\, M$ that extends continuously to $\partial M$ with zero boundary values.\smallskip

In case $\pma' M=\emptyset$, as it happens e.g. if $M$ is compact, we agree to define $k(x,\zeta) \equiv +\infty$. Note that, in this case,
\[
\CH^{0}_{-\Delta}(\inte\, M) = \emptyset
\]
thanks to the version of the Martin representation formula that can be found e.g. in  \cite[Chapter 8, Corollary 1.6]{Pinsky}.  In the next Lemma, we point out that, actually, the existence of bounded functions in  $\CH^{0}_{-\Delta}(\inte\, M)$ characterizes the $\D$-hyperbolicity of a manifold with boundary.
\begin{lemma}\label{lemma-criticality-positive}
 Let $(M,g)$ be a Riemannian manifold with smooth boundary $\partial M \not= \emptyset$. Then, $M$ is $\D$-parabolic if and only if $\CH^{0}_{-\Delta}(\inte\, M) \cap L^{\infty}(M) = \emptyset$.
\end{lemma}
\begin{proof}
We need only to show that $\CH^{0}_{-\Delta}(\inte\, M) \cap L^{\infty}(M) = \emptyset$ implies $\D$-parabolicity. This follows from the proof of Proposition \ref{DirParSubharmonic}. Indeed, let $u \in C^{\infty}(\inte\, M) \cap C^{0}(M)$ be a nontrivial bounded harmonic function satisfying $u=0$ on $\partial M$. Up to replacing $u$ with $-u$ we can suppose that $u_{\epsilon}(x) := (u-\epsilon)_{+}(x)  \not\equiv 0$ for some $0<\epsilon \ll1$. Since $u_{\epsilon} \geq 0$ is a bounded subharmonic function satisfying $ u_{\epsilon}=0$ on $\partial M$,
arguing as in the proof of  Proposition~\ref{DirParSubharmonic}
we get a bounded harmonic function $v$ such that $v=0$ on $\partial M$ and $0 \leq u_{\epsilon} \leq v$. In particular $v \geq 0$ is nontrivial and, hence, by the maximum principle $v>0$ on $\inte\, M$. This proves that $v \in \CH^{0}_{-\Delta}(\inte\, M)\cap L^{\infty}(M) \not=\emptyset$.
\end{proof}
With this preparation, we have the following interpretation of  the $\D$-parabolicity.
\begin{theorem}
 The Riemannian manifold $(M,g)$ with smooth boundary $\partial M \not=\emptyset$. If $M$ is $\D$-parabolic then
for every $\zeta \in \pma' M$, $k(x,\zeta)$ is unbounded. If $ \pma' M$ is at most countable, and $k(x,\zeta)$ is unbounded for every $\zeta \in \pma' M$, then $M$ is $\D$-parabolic.
\end{theorem}
\begin{proof}
Assume first that $M$ is $\D$-parabolic. Then, for each $\zeta \in \pma' M$, the corresponding Martin kernel $k(x,\zeta)$ must be unbounded for, otherwise, we would have $k(x,\zeta) \equiv 0$ on $\inte\, M$.\\
Conversely, assume that $\pma' M$  is at most countable and that the Martin kernels $k(x,\zeta)$ with poles at $\pma' M$ are unbounded.
By Lemma \ref{lemma-criticality-positive} we just need to prove that $\CH^{0}_{-\Delta}(\inte\, M) \cap L^{\infty}(M) = \emptyset$.
By contradiction, let $u \in \CH^{0}_{-\Delta}(\inte\, M) \cap L^{\infty}(M)$. By the representation formula (see, e.g. \cite[Chapter 8, Theorem 1.4, Corollary 1.6]{Pinsky})
\[
u(x)=\int_{\pma' M} k(x,\zeta) d\mu_u(\zeta),
\]
where $\mu_u$ is a probability measure supported on  $\pma' M$. Since the latter is at most countable, there exists $\zeta_o\in \pma' M$ such that $\mu_u(\{\zeta_o\})>0$. It follows that
\[
\mu_u(\{\zeta_o\}) k(x,\zeta_o)= \int_{\{\zeta_o\}}  k(x,\zeta) d\mu_u(\zeta) \leq u(x),
\]
on $M$ against the assumption that all Martin kernels with poles on $\pma' M$ are unbounded.
\end{proof}
\begin{remark}\label{KernelsVsParabolicity}{\rm
In the general case, the unboundedness of the Martin kernels $k(x,\zeta)$ for every $\zeta \in \pma' M$ does not seem to imply the $\D$-parabolicity of $M$. We are indebted to B. Devyver for having pointed out to us the following counterexample: Let  $B_2(0)\setminus B_1(0)\subset \mathbb{R}^2$ which we view as an (incomplete) manifold with boundary $\partial M=\partial B_1(0)$ and  Martin boundary $\partial B_2(0)\cup \partial B_1(0)$. According to \cite[Theorem 8.2.2]{K}, the Poisson (=Martin) kernels with poles at $\zeta\in \partial B_2(0)$ are bounded from below by $C(2-|x|)/|x-\zeta|^2$ and are therefore unbounded. Nevertheless $u=\log |x|$ is a bounded positive harmonic function on $M$ vanishing on $\partial M=\partial B_1(0)$.
}
\end{remark}

The theory we have developed so far in this section relies on the subcriticality of the Laplace-Beltrami operator (with Dirichlet boundary conditions) and on the corresponding Martin theory. Since these theories have been developed in the setting of general Schr\"odinger operators $P = -\Delta + V$ and, in the recent paper \cite{PiSa},  they have been extended to general mixed boundary conditions, it is natural to inquire whether the theory of $\D$-parabolic manifolds can be extended accordingly. Although, formally, the definition of $\D$-parabolic manifold, as well as the subsequent $L^{1}$-Liouville theory, seem to survive in the more general framework of subcritical operators with mixed boundary conditions, we feel that some work still has to be done. For instance, the criticality theory of \cite{PiSa} seems to include both the Dirichlet and the Neumann case but we know that the corresponding notions of parabolicity are distinct and subordinated. A unified theory that includes both would be important even at the conceptual level, and we aim at investigating this fascinating possibility in a separate article.

\section{Dirichlet $L^1$-Liouville property}\label{section-DL1Liouville}

In this section we  introduce the notion of Dirichlet $L^1$-Liouville property for a smooth manifold $M$ with nonempty boundary $\partial M$ and show how this property can be useful when one tries to find localized sufficient conditions to guarantee the global $L^1$-Liouville property.

\begin{definition}
Let $M$ be a smooth manifold with nonempty boundary $\partial M$. We say that $M$ is Dirichlet $L^1$-Liouville (shortly, $M$ is $\D$-$L^1$-Liouville) if every nonnegative $L^1$ solution of
\[
\left\lbrace
\begin{array}{rl}
\Delta u  \leq 0 & \text{in} \quad \inte\, M \\
u  = 0 & \text{on} \quad \partial M,
\end{array}
\right.
\]
vanishes identically on $M$.
\end{definition}

When the smooth manifold $M$ does not have  a boundary A. Grigor'yan \cite{Grigoryan_stochastic_harmonic} shows that the validity  of the $L^1$-Liouville property is equivalent to the nonintegrability  of the Green's kernel $G$ of $M$. A similar characterization holds also for the $\D$-$L^1$-Liouville property in a smooth manifold with nonempty boundary up to using the Dirichlet Green's kernel ${}^{\D}\!G$. In order to state precisely this result we give the next

\begin{definition}
The Dirichlet mean exit time ${}^{\mathcal{D}}\!E$  of $M$ is defined as the minimal positive solution of
\begin{equation}\label{DirichletMeanExitTime}
\left\lbrace
\begin{array}{cl}
\Delta {}^{\mathcal{D}}\!E = -1 & \text{in} \quad \inte\, M \\
{}^{\mathcal{D}}\!E=0 & \text{on} \quad \partial M.
\end{array}
\right.
\end{equation}
\end{definition}

As in the case of manifolds without boundary, the Dirichlet mean exit time  ${}^{\mathcal D}\!E$ can be constructed by an exhaustion procedure. Let $\{\Omega_k\}$ be an increasing exhaustion of $M$ by relatively compact (in $M$) open sets with smooth boundary $\overline{\partial_0 \Omega_k} $ which intersect transversally $\partial M$, and for every $k$ let ${}^{\mathcal{D}}\!E^{\Omega_k}$ be the solution of
\begin{equation}\label{E_k_Omega}
\Delta{}^{\mathcal{D}}\!E^{\Omega_k}=-1 \,\,\text{ in }\,\, \inte\,\Omega_k,\qquad {}^{\mathcal{D}}\!E^{\Omega_k}=0\,\,\text{ on }\,\, \partial\Omega_k.
\end{equation}
Then the sequence ${}^{\mathcal{D}}\!E^{\Omega_k}$ either diverges at every point in $\inte\,M$ or it converges monotonically to a function ${}^{\mathcal{D}}\!E$ which is clearly the minimal solution of \eqref{DirichletMeanExitTime}.

The above construction and the comparison principle yield the
following lemma, which will be useful in Section~\ref{StochVsL1} below.

\begin{lemma}
\label{UpBndMeanExitTime}
Let $M$ be a manifold with boundary $\partial M$, and let $E:M\to \rr$ satisfy
$\Delta E\leq -1$ in $\inte\, M$, $E\geq 0$ on $M$. Then ${}^{\mathcal{D}}\!E$ is finite and bounded above by $E$.
\end{lemma}

The relation between the Dirichlet mean exit time and the $\D$-$L^1$-Liouville property is given by the following version of Grigor'yan's result alluded to above.
\begin{theorem}\label{thm_D_L_1}
Let $M$ be a manifold with boundary $\partial M$. The following are equivalent:
\begin{enumerate}
\item[i)] $M$ is not $\D$-$L^1$-Liouville;
\item[ii)] The Dirichlet Green's kernel ${}^{\mathcal{D}}\!G(x,\cdot) $ is in $L^1(M)$ for any $ x \in M$;
\item[iii)] The Dirichlet mean exit time ${}^{\mathcal{D}}\!E$ is  finite everywhere and given  by
\[
{}^{\mathcal{D}}\!E(x) = \int_M {}^{\mathcal{D}}\!G(x,y)dy.\]

\end{enumerate}
\end{theorem}
\begin{proof}
The proof follows closely Grigor'yan's arguments in \cite{Grigoryan_stochastic_harmonic}.\\
$i) \Rightarrow ii)$ Let $u$ be a nonnegative function violating the $\D$-$L^1$-Liouville property. By the minimum principle $u$ is strictly positive in $\inte\, M$. Fix a
compact set $V \Subset \inte\, M$, a point $x\in V$ and let $C$ be a large constant satisfying
\begin{equation}\label{ineq_G_V}
{}^{\mathcal{D}}\!G(x,y) \leq Cu(y), \quad \forall \, y \in \partial V.
\end{equation}
Taking a smooth relatively compact exhaustion $\{\Omega_{k}\} \nearrow M$ with $V \subset \Omega_1$ and $\overline{\partial_0 \Omega_k}$ intersecting transversally $\partial M$, since $u$ is superharmonic and ${}^{\mathcal{D}}\!G^{\Omega_{k}}(x,\cdot)$ is harmonic on $M \backslash V$, again by the minimum principle, inequality \eqref{ineq_G_V} holds for ${}^{\mathcal{D}}\!G^{\Omega_{k}}(x,\cdot)$ and $u$ on $\Omega_k \backslash V$, and letting $k \rightarrow +\infty$ we get that \eqref{ineq_G_V} holds for all $y \in M \backslash V$. Hence, being ${}^{\mathcal{D}}\!G(x,\cdot)$ locally integrable, we obtain
\begin{eqnarray*}
\int_M {}^{\D}\!G(x,y)\,dy &=& \int_V {}^{\D}\!G(x,y)\,dy + \int_{M\backslash V} {}^{\D}\!G(x,y)\,dy \\
&\leq &\int_V {}^{\D}\!G(x,y)\,dy + C\int_{M\backslash V} u(y)\,dy < + \infty.
\end{eqnarray*}
$ii) \Leftrightarrow iii)$ Let $\Omega_k$ be an exhaustion as above and denote by ${}^{\mathcal{D}}\!G^{\Omega_{k}}$ the corresponding Dirichlet Green's kernel.
The function $\int_{\Omega_k}{}^{\mathcal{D}}\!G^{\Omega_{k}}(x,y)\,dy$ satisfies \eqref{E_k_Omega}, and by the  uniqueness of the Dirichlet problem in compact sets, it coincides with ${}^{\mathcal{D}}\!E^{\Omega_{k}}$. Since ${}^{\mathcal{D}}\!G^{\Omega_{k}}(x,y)\nearrow {}^{\mathcal{D}}\!G(x,y)$, passing to the limit we deduce that
\[
{}^{\mathcal{D}}\!E(x) = \lim_{k\rightarrow +\infty} {}^{\mathcal{D}}\!E^{\Omega_{k}}(x) =
\lim_{k\rightarrow +\infty}\int_M {}^{\mathcal{D}}\!G^{\Omega_{k}}(x,y)\chi_{_{\Omega_k}}dy
 = \int_{M} {}^{\D}\!G(x,y)\,dy.
\]
$ii) \Rightarrow i)$
Assume that ${}^{\D}\!G(x,\cdot) \in L^1(M)$. Then the function
\[
u = \min\{{}^{\D}\!G(x,\cdot), 1\}
\]
is  nonnegative, nonconstant, superharmonic and integrable, thus violating the $\D$-$L^1$-Liouville property.
\end{proof}
\begin{remark}\label{exhaustion_int_mean_exit_time}\rm
Note that since the Dirichlet Green's kernel of a manifold $M$ coincides with the Green's kernel of $\intM$, the Dirichlet mean exit time is, in fact, the mean exit time of $\intM$. This means that ${}^{\mathcal{D}}\!E$ can be constructed by means of an exhaustion  $\{\Omega_k\} $ consisting of open sets with smooth boundary which are relatively compact in $\intM$.
\end{remark}

The first consequence of Theorem \ref{thm_D_L_1} is a relationship between the global $L^1$-Liouville property of a manifold and the $\mathcal{D}$-$L^1$-Liouville property for an open submanifold with smooth boundary.

\begin{corollary}\label{thm_comparison_dirichlet}
Let $(N,h)$ be a manifold with empty boundary and let $M\subset N$ be the closure of a smooth open submanifold. Then $E^N\geq {}^{\mathcal{D}}\!E^M$. In particular, if $M$ is $\mathcal{D}$-$L^1$-Liouville, then $N$ is  $L^1$-Liouville. Moreover, if there exists a sequence of smooth open submanifolds $M_k$ such that ${}^{\mathcal{D}}\!E^{M_k}(\bar x)\to +\infty$ as $k\to \infty$,  then $N$ is $L^1$-Liouville.
\end{corollary}
\begin{proof}
Without loss of generality we can assume that $N$ is hyperbolic. In this case, $G^N$ exists and satisfies \eqref{ineq_G_D}, that is, $G^N \geq {}^{\D}\!G^{M}$. Hence, integrating yields $E^N\geq {}^{\mathcal{D}}\!E^M$, and the conclusion follows by Theorem \ref{thm_D_L_1}. Similarly, if there exists a sequence $M_k$ containing a fixed point $\bar x$, such that ${}^{\mathcal{D}}\!E^{M_k}(\bar x)\to +\infty$ then $E^N(\bar x)\geq {}^{\mathcal{D}}\!E^{M_k}(\bar x)\to +\infty$ and again $N$ is $L^1$-Liouville.
\end{proof}

\begin{remark}
{\rm
It is worth to observe that the first assertion in the statement of the corollary does not follow immediately from the maximum principle due to the noncompactness of the open submanifold $M$.
}
\end{remark}
As it might be expected, in view of what happens with  stochastic completeness (cf. \cite{Bessa-Bar, BPS-Pota}), the $\D$-$L^1$-Liouville property is an asymptotic property. We are going to justify this claim via two interesting results. To begin with, we show with a very direct argument that manifolds which are isometric outside a compact set share the same behaviour with respect to the $L^{1}$-Liouville property.

\begin{proposition}\label{prop-L1-ends}
Let $M_1$, $M_2$ be two Riemannian manifolds with nonempty boundaries $\partial M_1$ and $\partial M_2$, respectively. Assume that there exists an isometry between $M_1 \backslash K_1$ and $M_2 \backslash K_2$, where $K_1 \subset \inte\,M_1$ and $K_2 \subset \inte\,M_2$ are compact sets. Then, $M_1$ is $\D$-$L^1$-Liouville if and only if so is $M_2$. The same conclusion holds true in the case of manifolds without boundary.
\end{proposition}
\begin{proof}
Suppose by contradiction that $M_1$ satisfies the $\D$-$L^1$-Liouville property, but $M_2$ does not. In this case there exists a nonnegative superharmonic function $ u \in L^1(M_2)$ with $u = 0$ on $\partial M_2$. Let $\phi : M_1 \backslash K_1 \to M_2 \backslash K_2$ be an isometry, and let $\Omega_1, \Omega_2 \subset M_1$ be relatively compact open sets such that $K_1 \Subset \Omega_1 \Subset \Omega_2$.

Let $u_1 : M_1 \backslash K_1 \to \rr $ be a function given by $ u_1 = u \circ \phi$. We know that $u_1$ is a nonnegative superharmonic integrable function. In order to extend this function on all of $M_1$, we consider
$$ c \doteq \min_{\bar \Omega_2\backslash \Omega_1} u_1 > 0$$
and define
\[v_1 \doteq
\left\{
\begin{array}[c]{l}
\min\{u_1 ,c\} \ \ \text{in} \ \ M_1 \backslash \bar\Omega_1 \\
c \ \ \text{on} \ \ \Omega_2.
\end{array}
\right.
\]
The function $v_1$ is well defined because $v_1 \equiv c$ on $\bar\Omega_2\backslash \Omega_1$. Furthermore, $v_1$ is a nonnegative superharmonic function because it is superharmonic in $M_1 \backslash \bar\Omega_1 $ and in $\Omega_2$, open sets covering $M_1$. Finally, since $\phi(\partial M_1) = \partial M_2$ we have a contradiction because $v_1$ is clearly integrable at infinity, vanishes at $\partial M_1$, and is not identically zero.
\end{proof}

Next, we show that the validity of the $L^{1}$-Liouville property of a manifold without boundary depends on the $\D$-$L^{1}$-Liouville property of (at least) one of its ends. This can be considered as a converse of Corollary \ref{thm_comparison_dirichlet} when the open submanifold $M$ has compact complement  in $N$.

\begin{theorem}\label{asymptotic_invariance_L1}
Let $(N,h)$ be a manifold with empty boundary $\partial N = \emptyset$. The following assertions hold.
\begin{enumerate}
\item[i)] Let $M$ be a smooth open submanifold in $N$ such that $N\backslash M$ is compact. If $N$ is $L^1$-Liouville, then $M$ is $\D$-$L^1$-Liouville.
\item[ii)] Let $E$ be an end of $N$. Then $E$ is $\D$-$L^1$-Liouville if and only if its Riemannian double $\D(E)$ is $L^1$-Liouville.
\item[iii)] $N$ is $L^1$-Liouville if and only if at least one of its ends is $\D$-$L^1$-Liouville.
\end{enumerate}
\end{theorem}
\begin{proof}
The  proof of $i)$ is similar to that of $i) \Rightarrow ii)$ in Theorem \ref{thm_D_L_1}. Let $G^N$ be the Green's kernel of $N$. By assumption
$$ \int_N G^N(x,y)\,dx = + \infty \ \ \forall \, y \in N$$
and recall that $G^N$ is obtained by an exhaustion procedure
$$ G^N = \lim_{k} {}^{\mathcal{D}}\!G^{\Omega_k},$$
where ${\Omega_k}$ is a sequence of relatively compact open domains with smooth boundary. Since $N\backslash M$ is compact we may assume that $N\backslash M\Subset \Omega_1$. Let $u \geq 0$ satisfy $\Delta u \leq 0$ in $M$ and $u=0$ on $\partial M$. Since $u > 0$ in $M$ and $\partial M \subset \Omega_1$,
$$ \inf_{\partial\Omega_1}u = c_1 > 0$$
and having fixed $y_0 \in \Omega_1$ we have
$$ \sup_{\partial\Omega_1}G^N(x,y_0) = c_2 < +\infty.$$
Thus, there exists $\lambda > 0$ such that $u(x) \geq \lambda G^N(x,y_0)$ for any $x \in \partial\Omega_1$, and since $G^N \geq {}^{\mathcal{D}}\!G^{\Omega_k}$ we have $u \geq \lambda \,{}^{\mathcal{D}}\!G^{\Omega_k}$ on $\partial\Omega_k \cup \partial\Omega_1$. By the maximum principle $u\geq \lambda \,{}^{\mathcal{D}}\!G^{\Omega_k}$ in $\Omega_k \backslash \Omega_1$ and passing to the limit $u \geq \lambda G^N$ in $N\backslash \Omega_1$. Integrating, and recalling that $\Omega_1$ is compact and that $G^N$ is locally integrable, we conclude
$$ \int_M u(x)\,dx \geq \int_{M\backslash\Omega_1}u(x)\,dx \geq \lambda \int_{M\backslash\Omega_1}G^N(x,y_0)\,dx = +\infty.$$
A similar argument and Corollary \ref{thm_comparison_dirichlet} prove $ii)$. For $iii)$ we know that if at least one end is $\D$-$L^1$-Liouville, then $N$ is $L^1$-Liouville by Corollary \ref{thm_comparison_dirichlet}. Conversely, suppose that no end of $N$ satisfies the $\D$-$L^1$-Liouville property, that is, there exists a compact set $K \subset N$ such that $N\backslash K = \cup_{i=1}^{m}E_i$ where $E_i$ are ends of $N$ and ${}^{\D}\!E^{E_i}$ are finite functions, for any $i=1,...,m$. Since $\partial E_i \cap \partial E_j = \emptyset$ for any $i\not=j$, setting  ${}^{\D}\!E^{E_i}\equiv0$ outside of $E_i$, it is easy to verify that ${}^{\D}\!E^{N\backslash K} = {}^{\D}\!E^{E_1} + \cdots +{}^{\D}\!E^{E_m}$ and the desired contradiction follows from part $i)$.
\end{proof}

\section{Stochastic completeness vs $L^1$-Liouville property} \label{StochVsL1}

As we have seen in the introduction, A. Grigor'yan \cite{Grigoryan_stochastic_harmonic}, applying the version of Theorem \ref{thm_D_L_1} for manifold without boundary, proved that every stochastically complete manifold is $L^1$-Liouville. Furthermore, the two concepts are in fact equivalent for the large class of rotationally symmetric manifolds, \cite{BPS-Pota}. In this section we will present some examples in any dimension showing that, in general, the two concepts are not equivalent. This completes the picture in  \cite{BPS-Pota} where only the case of $2$-dimensional surfaces was considered.\smallskip

Since the $L^1$-Liouville property is equivalent to the divergence of the integral
$$ \int_M G(x,y)\,dy = \int_M \int_{0}^{\infty} p(t,x,y)\,dt\,dy ,$$
a natural approach to investigate its validity is by means of heat kernel, respectively Green's kernel, estimates.

In the first example we employ parabolic arguments to show the divergence of the above integral.
We will make use of the following estimates obtained in this setting by A. Grigor'yan and L. Saloff-Coste in \cite{Grigoryan-Saloff-Coste}.

We recall that a manifold satisfies the parabolic Harnack inequality (PHI) if there exists a constant $C_0$ such that any nonnegative solution $u$ of the heat equation $ \partial_{t}u = \bigtriangleup u$ in any cylinder $  Q = (\tau,\tau + T)\times B(x,r)$ with $ T=r^2$ and $ \tau \in (-\infty,+\infty)$, satisfies
\begin{equation}\tag{PHI}
\sup_{Q_{-}}u \leq C_{0}\inf_{Q_{+}}u ,
\end{equation}
where ${Q_{-}} = (\tau+\frac{T}{4},\tau + \frac{T}{2})\times B(x,\frac{r}{2})$ and $ Q_{+} = (\tau+\frac{3T}{4},\tau+T)\times B(x,\frac{r}{2})$.

It is well-known (see \cite[Theorem 5.5.1]{Saloff-Coste_book}) that (PHI) is equivalent to the simultaneous validity of the doubling and the  scale-invariant Poincar\'e inequalities, and in a complete manifold it is implied by nonnegative Ricci curvature.

\begin{theorem}\cite[Theorem 3.3]{Grigoryan-Saloff-Coste}\label{thm_grigoryan_saloff}
Let $M$ be a complete nonparabolic manifold, $\partial M = \emptyset$. Assume that the parabolic Harnack inequality (PHI) holds on $M$ and let $K$ be a compact set in $M$. Then there exists $\delta >0$ and, for each $t_0 > 0$, there exist positive constants $C$ and $c$ such that, for all $ t> t_0$ and all $x,y \in \Omega = M \backslash K_{\delta}$, $K_\delta$ being the $\delta$-neighborhood of $K$,
\begin{equation}
p^{\Omega}(t,x,y) \geq \frac{c}{V(x,\sqrt{t})}\exp\left(-C\frac{d^{2}(x,y)}{t}\right).
\end{equation}
\end{theorem}

We now present an example in arbitrary dimension of a stochastically incomplete $L^1$-Liouville manifold with at least two ends.

\begin{example}\label{ex_1_L1_stoc_incomplete}\rm
Let $M$ be a Riemannian manifold, $\partial M = \emptyset$, defined by the connect sum $M = M_{1} \# M_{2}$, where
\begin{enumerate}
\item[$-$] $ M_{1} $ is a complete nonparabolic Riemannian manifold which supports (PHI);\vspace{0.2cm}
\item[$-$] $ M_{2} $ is a geodesically complete and stochastically incomplete Riemannian manifold.\vspace{0.1cm}
\end{enumerate}
We claim that $M$ is a stochastically incomplete $L^{1}$-Liouville manifold.

Indeed, since $M$ is defined by a connected sum and $M_2$ is stochastically incomplete then $M$ is stochastically incomplete (see \cite{Bessa-Bar} or more generally \cite{BPS-Pota}). Now, let $ K \subset M_{1}$ be a compact subset containing the disk along whose boundary the gluing of $ M_{1}$ with $ M_{2}$ is performed. Denote by $ \Omega = M_1 \backslash K_{\delta}$, the subset obtained in Theorem \ref{thm_grigoryan_saloff}, and consider the Dirichlet heat kernel $ {}^{\mathcal{D}}\!p^{\Omega}$, the Dirichlet Green's kernel ${}^{\mathcal{D}}\!G^{\Omega}$, and the Dirichlet mean exit time ${}^{\mathcal{D}}\!E^{\Omega}$ of $\Omega$. Being $M_1$ nonparabolic, from Theorem \ref{thm_grigoryan_saloff} we have for $ x \in \inte\,\Omega$
\begin{eqnarray*}
{}^{\mathcal{D}}\!E^{\Omega}(x) &=& \int_{\Omega}{}^{\mathcal{D}}\!G^{\Omega}(x,y)dy \\
&=& \int_{\Omega}\int_{0}^{\infty}{}^{\mathcal{D}}\!p^{\Omega}(t,x,y)dt dy \\
&\geq & \int_{\Omega}\int_{t_0}^{\infty}\frac{c}{V_{1}(x,\sqrt{t})}\exp\left(-C\frac{d^{2}(x,y)}{t}\right)dt dy \\
&\geq & \int_{\Omega \cap B(x,\sqrt{t})}\int_{t_0}^{\infty}\frac{c}{V_{1}(x,\sqrt{t})}\exp\left(-C\frac{t}{t}\right)dt dy \\
&=& c \exp(-C)\int_{t_0}^{\infty}\frac{V_{1}(\Omega \cap B(x,\sqrt{t}))}{V_{1}(x,\sqrt{t})}dt \\
&\geq & c \exp(-C)\left(\int_{t_1}^{\infty}dt - \int_{t_1}^{\infty}\frac{V_{1}(K_{\delta})}{V_{1}(x,\sqrt{t})}dt\right) \\
&=& +\infty .
\end{eqnarray*}
The conclusion then follows by Corollary \ref{thm_comparison_dirichlet}.
\end{example}

The following is an alternative, slightly more general, construction based on Theorem \ref{asymptotic_invariance_L1}. Indeed, suppose that $M_1$ is $L^1$-Liouville with one end, for instance $M_1$ is a stochastically complete model manifold (see, e.g., \cite{BPS-Pota}) and let $M_2$  and  $M=M_1\# M_{2}$ be as above. Let also $K$ be a compact set with smooth boundary, containing the disks along which $M_1$ and $M_2$ have been glued, and denote by $E_1$ the end in $M\backslash K$  isometric to the corresponding end of $M_1$. Since $M_1$ is $L^1$-Liouville, by Theorem \ref{asymptotic_invariance_L1} $i)$, $E_1$ is $\mathcal D$-$L^1$-Liouville, and therefore $M$ is $L^1$-Liouville by Theorem \ref{asymptotic_invariance_L1} $iii)$. On the other hand, as we have already observed, $M$ is stochastically incomplete due to the presence of the  $M_{2}$ summand.\smallskip

In the next example we construct a stochastically incomplete $L^1$-Liouville manifold with only one end.

\begin{example}\label{ex_2_L1}\rm
Let $h$ be any Riemannian metric on $\rr^{m}$ which makes $(\rr^{m}, h)$ stochastically incomplete and such that $h = g_{\rr^{m}}$ on the upper half space $\bar \rr^{m}_{+} \doteq \{(x^1,\dots,x^m) \in \rr^{m} : x^m \geq 0\}$. For instance, we can let $g$ be the Riemannian metric given in polar coordinates by
$g = dr^2 + g_1(r)^2d\theta^2,$ with $g_1(r)= e^{r^3}$ for $r \gg 1$, $g_1$ smooth and satisfying $g_1(0) = 0$ and $g_{1}'(0)=1$
and define $h$ as the convex sum $h = (1-\xi)g + \xi g_{\rr^m}$ where $\xi$ is a cut-off function depending only on $x^m$ and satisfying
$$\xi(x) = \left\lbrace
\begin{array}{ll}
1 & \text{if} \ \ x^m \geq 0 \\
0 & \text{if} \ \ x^m \leq -1.
\end{array}\right.
$$
Arguing as in \cite[p. 318] {BPS-Pota} it is easy to construct a function which violates the weak maximum principle at infinity, and therefore $M$ is not stochastically complete.\smallskip

We claim that $(\rr^{m}, h)$ is $L^1$-Liouville. According to Corollary \ref{thm_comparison_dirichlet} it is enough to show that the Euclidean half space $\bar \rr^{m}_{+}$ is $\D$-$L^1$-Liouville. We shall provide three different ways to prove this fact. The first and the second rely strongly on the explicit expression of the heat and the Green's kernel of half spaces. With the third, elegant method we illustrate how the machinery based on a global comparison principle works.\smallskip

\noindent{\bf Heat kernel estimates:}
By making use of the expression of the Dirichlet heat kernel of $\bar \rr^{m}_{+}$ P. Gyrya and L. Saloff-Coste \cite{Gyrya_Saloff-Coste} give the following lower bound
\begin{equation*}
{}^{\D}\!p^{{\bar \rr^{m}_{+}}}(t,x,y)\geq \frac{Cx^{m}y^{m}}{t^{\frac{n}{2}}(x^m + \sqrt{t})(y^m + \sqrt{t})}\exp\left(-\frac{c\vert x-y\vert^2}{4t}\right).
\end{equation*}
Arguing as in Example \ref{ex_1_L1_stoc_incomplete} we achieve the desired conclusion.\smallskip

\noindent{\bf Green's kernel estimates:}
Recalling the expression of the Dirichlet Green's kernel of $\bar \rr^{m}_{+}$ given in Example \ref{ex_euclid_half_space_2}, we have, for $m\geq 3$
\begin{eqnarray*}
{}^{\D}\!E^{\bar \rr^{m}_{+}}(x) &=& \int_{\bar \rr^{m}_{+}}{}^{\D}\!G^{\bar \rr^{m}_{+}}(x,y)\,dy \\[0.1cm]
&=& C_m\int_{\bar \rr^{m}_{+}}\left(\frac1{|x-y|^{m-2}} - \frac 1{|x-y'|^{m-2}}\right)\,dy
\end{eqnarray*}
where we have set $y' \!=\! (y^1,\ldots,y^{m-1},-y^m)$.
Choosing $x\!=\!(0,\dots,0,1)$, passing to cylindrical coordinates $(r,\theta, x^m)$ and applying Lagrange's theorem, the integrand can be written as
\[
\frac 1{[r^2+(x^m-1)^2]^{(m-2)/2}}-\frac 1{[r^2+(x^m+1)^2]^{(m-2)/2}}=C{\xi}[\xi^2+r^2]^{-m/2}
\]
with $x^m-1<\xi<x^m+1$ and it is therefore estimated from below by
\[
C(x^m-1) [(x^m+1)^2+r^2]^{-m/2}\geq \frac 12(x^m+1) [(x^m+1)^2+r^2]^{-m/2} \ \ \text{ for } \, x^m\geq 3.
\]
We may therefore estimate
\begin{equation*}
{}^{\D}\!E^{\bar \rr^{m}_{+}}(x) =
C_m \int_{0}^\infty r^{m-2} dr\int_3^{\infty} (x^m+1) [(x^m+1)^2+r^2]^{-m/2}=+\infty,
\end{equation*}
and $\bar \rr^{m}_{+}$ is $\mathcal{D}$-$L^1$-Liouville. Similar computations yield the conclusion if $m=2$.

\noindent{\bf Elliptic argument:}
We describe now an elliptic argument very much in the spirit of the  general results that follow. Note that by using the machinery so far developed it allows to obtain the required conclusion with a minimal amount of explicit computations.

Given $k>1$, consider the slab $\Omega_k = \{x \in \rr^{m}_{+} : \frac{1}{k} \leq x^m \leq k\}$. It is clear that $E_k(x) = \frac 1 2 (x^m - \frac 1 k)(k - x^m)$ is a positive, bounded solution of
\[\left\lbrace
\begin{array}{rrl}
\Delta E_k = & -1 & \text{in} \ \ \inte\, \Omega_k \\
E_k = & 0 & \text{on} \ \ \partial \Omega_k .
\end{array}
\right.
\]
Moreover, since $\phi_k(x) =  E_k(x) + \varepsilon\vert x\vert^2$ is a  Khas'minskii function for $\Omega_k$ whenever  $\varepsilon$ is sufficiently small, by Lemma \ref{Khas-D-parabolicity} $\Omega_k$ is $\D$-parabolic. Therefore the function $E_k$ is its (finite) Dirichlet mean exit time. Now, fixing $ x \in \inte\, \Omega_1$, we see that  $E_{k}(x) \to +\infty$ as $k \to +\infty$. Therefore, Corollary \ref{thm_comparison_dirichlet} implies that $\bar \rr^{m}_{+}$ is $\mathcal{D}$-$L^1$-Liouville.
\end{example}

\section{Localized geometric conditions for the $L^1$-Liouville property}\label{section-localized}

According to the theory developed so far, in view of the above examples, and as we have already observed in the Introduction of the paper, it is natural to guess that the validity of the $L^1$-Liouville property depends only on geometric conditions confined on a sufficiently large domain of a manifold. This section aims to confirm this intuition by considering three different situations: the first one, very general, involves a (asymptotically) nonnegative Ricci curvature condition on a geometric half space; the second  concerns with the complement of half space in hyperbolic situations, and the third deals with warped product cones and depends on a spectral assumption.

\subsection{Nonnegative Ricci curvature on a half space}

Inspired by Example \ref{ex_2_L1}, we  first investigate the role of lower Ricci curvature bounds in a geometric half space. Essentially, we are going to work with nonnegative Ricci curvature, but a certain amount of negativity is also allowed. For a geometric half space we mean the following.

\begin{definition}\label{def01}
Let $\gamma:[0,+\infty)\rightarrow M$ be a geodesic ray, parametrized by arc-length, into the complete Riemannian manifold $\left(M,g\right)$. The half space $M^{+}$ with respect to $\gamma$ is the domain
\[
\displaystyle{ M^{+}=\cup_{t>0}B_{t}\left(\gamma\left(  t\right)  \right)}  ,
\]
where $B_{t}\left(p\right)$ is the open metric ball of radius $t>0$ centered at $p$.
\end{definition}

\begin{theorem}\label{thm_L1_nonnegative_curv}
Assume that $\left(M,g\right)$ is a complete Riemannian manifold satisfying
\begin{equation}\label{RicLB}
\mathrm{Ric} \geq - \frac{(m-1)B^2}{1+r_{\gamma(0)}^2(x)}
\end{equation}  on the half space $M^{+}$ with respect to some geodesic ray $\gamma$, for some $0\leq B< \frac{\sqrt m}{m-1}, $ where $r_{\gamma(0)}$ denotes the distance function in $M$ from $\gamma(0)$. Then $\left(M,g\right)  $ is $L^{1}$-Liouville.
\end{theorem}
\begin{proof}
We begin with the  more general assumption that for every positive integer $k$  the following inequality holds
\begin{equation}\label{ineq_laplacian_r_k}
\Delta r_{k}(x) \leq (m-1)\frac{{\sigma'}_k}{\sigma_k}(r_{k}(x)) \ \ \ \forall\, x \in B_{t_k}(\gamma(t_k)),
\end{equation}
where $t_k$ is an increasing sequence with $t_1>1,$ $ r_{k}(x)=d(x,\gamma(t_k))$ is the distance function from $ \gamma(t_k)$ and $ \sigma_k : [0,+\infty) \to [0,+\infty) $ is the warping function of an $m$-dimensional model manifold $M_{\sigma_k}$.

For each positive integer $k$ define the function $F_{k}: [0,t_k-1) \to [0,+\infty)$ by
\begin{equation*}
F_{k}(r) = \int_{r}^{t_k-1}\frac{\int_{0}^{t}\sigma^{m-1}_{k}(s)ds}{\sigma^{m-1}_{k}(t)}dt .
\end{equation*}
A simple computation shows that
\begin{equation*}
F_{k}^{'}(r) = - \frac{\int_{0}^{r}\sigma_{k}^{m-1}(s)ds}{\sigma_{k}^{m-1}(r)} < 0,
\end{equation*}
and, by assumption \eqref{ineq_laplacian_r_k}, the composition with the function $ r_k : B_{t_k}(\gamma(t_k)) \to [0,+\infty) $ satisfies
\begin{eqnarray*}
\Delta F_{k}(r_{k}(x)) &=& F_{k}^{''}(r_{k}(x)) + F_{k}^{'}(r_{k}(x))\Delta r_{k}(x) \\
&\geq & F_{k}^{''}(r_{k}(x)) + (m-1)F_{k}^{'}(r_{k}(x))\frac{{\sigma'}_k}{\sigma_k}(r_{k}(x)) \\
&=& \Delta_{\sigma_k}F_{k}(r_{k}(x)) \\
&=& -1 .
\end{eqnarray*}
Furthermore, since $ F_{k} = 0 $ on the boundary of $ B_{t_k-1}(\gamma(t_k))$, by the maximum principle
\begin{equation}\label{comparison01}
F_{k}(r_{k}(x)) \leq {}^{\D}\!E_{k}(x),
\end{equation}
where $ {}^{\D}\!E_{k}$ denotes the Dirichlet mean exit time of $ B_{t_k-1}(\gamma(t_k))$.

To illustrate the idea, suppose that $ \mathrm{Ric}\geq 0$ on  $M^+$, i.e., \eqref{RicLB} holds with $B=0.$ By the Laplacian comparison theorem we can take $ \sigma_k(s) = s$ for any $k$ and compute
\begin{equation*}\label{eqftk}
F_{k}(r_k) = \frac{(t_{k}-1)^{2}}{2m} - \frac{r_k^2}{2m}.
\end{equation*}
Therefore, by \eqref{comparison01}
\begin{equation*}
{}^{\D}\!E_{k}(x) \geq \frac{(t_{k}-1)^{2}}{2m} - \frac{r_{k}^{2}(x)}{2m},
\end{equation*}
whence, chosing $ x_1 = \gamma(2) \in B_{t_k-1}(\gamma(t_k))$ for every $ k $, we have $ r_{k}(x_1) = t_k - 2 $ and
\begin{eqnarray*}
{}^{\D}\!E_{k}(x_1) &\geq & \frac{(t_{k}-1)^{2}}{2m} - \frac{r_{k}^{2}(x_1)}{2m} \\
&=& \frac{(t_{k}-1)^{2}}{2m} - \frac{(t_k - 2)^2}{2m}\\
&=& \frac{2t_k - 3}{2m}.
\end{eqnarray*}
Letting $k \to +\infty$, we deduce that  $ {}^{\D}\!E_{k}(x_1) \to + \infty$, and by Corollary \ref{thm_comparison_dirichlet} $M$ is $L^1$-Liouville.

Note that the main ingredient in the above argument is
\begin{equation}\label{F_k_integral}
F_{k}(t_k - 2) = \int_{t_k - 2}^{t_k-1}\frac{\int_{0}^{t}\sigma_{k}^{m-1}(s)ds}{\sigma_{k}^{m-1}(t)}dt \to +\infty \ \ \text{as} \ \ t_k \to +\infty.
\end{equation}
Thus, if we assume that $\mathrm{Ric} \geq - \frac{(m-1)B^2}{1+r_{\gamma(0)}^2(x)}$ on $M^+$, for some constant $B$, then the triangle inequality implies that
 \[
 \mathrm{Ric} \geq
 \frac{B^2}{1+(t_k - r_{k})^2}, \quad \text{ for every  } \, x \in B_{t_k}(\gamma(t_k)),
 \]
 and, again by the Laplacian comparison theorem (cf. \cite[Theorem 2.4]{PRS_vanishing}),
  inequality \eqref{ineq_laplacian_r_k} holds with $\sigma_k \in C^2([0,+\infty))$ which solves
\begin{equation}\label{dir_prob_sigma_k}
\left\lbrace\begin{array}{l}
\sigma''_k -\frac{B^2}{1+(t_k - r_{k})^2}\sigma_k = 0 \\
\sigma_k(0)=0, \sigma'_k(0)=1.
\end{array}
\right.
\end{equation}
A  computation (see, e.g. \cite{BS-preprint}) shows that the solution $\sigma_k$ of \eqref{dir_prob_sigma_k} is explicitly given by
$$\sigma_k(s) = \frac{1}{2\bar B}t_{k}^{\frac{1}{2}+\bar B}(t_k - s)^{\frac{1}{2}-\bar B}\left[1-(1-\frac{s}{t_k})^{2\bar B}\right],$$
where $\bar B < \frac{\sqrt{4B^2 +1}}{2} \geq \frac{1}{2}$. Assume first that $s \in [0,t_{k}/2]$ and set $\lambda = s/t_k \in [0,1/2]$. From the mean value theorem, the function $\rho(\lambda) = 1 -(1-\lambda)^{2\bar B}$ satisfies $\rho(\lambda) = \rho'(\bar \lambda)\cdot\lambda$ for some $\bar \lambda \in (0,\lambda)$, and $\rho'(\bar \lambda) \geq 2\bar B(1-1/2)^{2\bar B -1} \doteq 2\bar B d_1 > 0.$ Thus, since $1/2 - \bar B \leq 0$, $ \sigma_k(s) \geq d_1 s$ for $s \in [0,\frac{t_k}{2}].$ When $s \in [t_k/2,t_k]$ we easily have that $ \sigma_k(s) \geq d_2 s,$ with $d_2 > 0$ depending only on $\bar B$. Hence, $\sigma_k(s) \geq \bar d s$ for $s \in [0,t_k]$  with $ \bar d= \min\{d_1,d_2\}$. Therefore, we compute
\begin{eqnarray*}
F_{k}(t_k - 2) = \int_{t_k - 2}^{t_k - 1}\frac{\int_{0}^{t}\sigma_{k}^{m-1}(s)ds}{\sigma_{k}^{m-1}(t)}\,dt &\geq & \int_{t_k - 2}^{t_k - 1}\frac{\frac{\bar d}{m}t^m\,dt}{(\frac{1}{2\bar B})^{(m-1)}t_{k}^{(m-1)(\frac{1}{2}+\bar B)}(t_1/2)^{(m-1)(\frac{1}{2}-\bar B)}} \\
&\geq & c \int_{t_k - 2}^{t_k - 1}\frac{(t_k - 2)^m\,dt}{t_{k}^{(m-1)(\frac{1}{2}+\bar B)}} \\
&\geq & c \frac{(t_k - 2)^m}{t_{k}^{(m-1)(\frac{1}{2}+\bar B)}},
\end{eqnarray*}
and conclude that \eqref{F_k_integral} holds since $\bar B < \frac{m+1}{2(m-1)}$ by the assumption that $B<\sqrt m/(m-1)$.
\end{proof}

\subsection{Pinched negative curvature in the complement of a hyperbolic half space}
In contrast with what happens in the case of nonnegative curvature, in the setting of negative curvature half spaces are not large enough to satisfy the $\D$-$L^1$-Liouville property, and therefore, to conclude the validity of the global $L^1$-Liouville property. We are going to show that the appropriate domains are represented by the complement of a half space.\smallskip

Throughout this subsection we consider $M$ to be an $m$-dimensional complete, simply connected Riemannian manifold satisfying, for some $B \geq A > 0$, the pinched curvature assumptions
\begin{equation}\label{ch-conditions}
\mathrm{Ric} \geq -(m-1) B^2 \quad \text{and} \quad \sect \leq -A^2.
\end{equation}
Obviously, these conditions imply that the (Cartan-Hadamard) manifold $M$ has a lower sectional curvature bound.

Let $\gamma : [0 , +\infty) \to M$ be a fixed unit-speed ray and let $b_{\gamma} : M \to \rr$ be the associated Busemann function. By definition
\[
b_{\gamma} (x) = \lim_{t\to + {\infty}} r_{t}(x) - t,
\]
where $r_{t}(x)=\mathrm{dist}_{M}(x,\gamma(t))$.
Then, the open horoball of $M$ with respect to $\gamma$ is defined by
\[
\cb_{\gamma}(R) = \{x \in M : b_{\gamma}(x) < R\}.
\]
Note that $\cb_{\gamma}(R)$ is a half space with respect to the unit-speed ray $\gamma$. Namely,
 \[
 \cb_{\gamma}(R) = \cup_{t \geq 0} B_{t+R}(\gamma(t)).
 \]
The verification is straightforward using the very definition of Busemann function.

\begin{example}\label{ex_horoball_D_parab}\rm
The closed  horoball $\bar\cb_{\gamma}^{2}(R)$ of the $2$-dimensional hyperbolic space $\mathbb{H}^{2}$ is $\mathcal{D}$-parabolic, regardless of the fact that its volume growth is exponential. Indeed, in the half plane model of the hyperbolic plane, $\bar \cb_{\gamma}^{2}(R)$ is realized as a closed half plane. Now, it is immediate to see that, in the $2$-dimensional setting, $\mathcal{D}$-parabolicity is invariant under conformal maps. Since $\bar \cb_{\gamma}^{2}(R)$ is conformal to the flat half plane $\bar \Pi^{+} $, we are reduced to prove that $\bar \Pi^{+}$ is $\mathcal{D}$-parabolic. To this end, observe that the volume growth of the flat manifold with boundary $\bar \Pi^{+}$ is at most quadratic. It follows from Proposition \ref{vol-growth-D-parabolicity} that $\bar \Pi^{+} $ is $\N$-parabolic, and therefore $\mathcal{D}$-parabolic. Hence $\bar \cb_{\gamma}^{2}(R)$ is a $\mathcal{D}$-parabolic manifold, as claimed.
\end{example}

The same conclusion holds in general dimensions  for  Cartan-Hadamard manifolds satisfying \eqref{ch-conditions}, but we need a completely different argument.
\begin{proposition}
Let $M$ be a Cartan-Hadamard manifold satisfying \eqref{ch-conditions}. Then the closed horoball $\bar \cb_{\gamma}(R)$ of $M$ with respect to the geodesic ray $\gamma$ is $\mathcal{D}$-parabolic.
\end{proposition}
\begin{proof}
 It is an application of Lemma \ref{Khas-D-parabolicity}. Indeed, assume $R > 1$ and let $\phi : \bar \cb_{\gamma}(R) \backslash B_{1}(\gamma(0)) \to \rr$ be the smooth function defined by
\[
\phi(x) = ( R - b_{\gamma}(x) ) + c\cdot r_{0}(x)
\]
where $c >0$ is a constant to be specified. Then, $\phi$ is positive and proper. Moreover, by the Hessian and Laplacian comparison theorems (see e.g. \cite {BLPS} for the Hessian estimate of $b_{\gamma}$),
\[
\Delta \phi \leq -(m-1)A + c \cdot  (m-1)B\mathrm{coth}(Br_{0}),
\]
and choosing $c$ sufficiently small
\[
\Delta \phi \leq 0 \quad \text{on }\,\cb_{\gamma}(R) \backslash B_{1}(\gamma(0)).
\]
Now, the desired conclusion  follows from Lemma \ref{Khas-D-parabolicity}.
\end{proof}

As stated at the beginning of this subsection the closed horoball $\bar \cb_{\gamma}(R)$ is not $\D$-$L^1$-Liouville. Taking into account Theorem \ref{thm_D_L_1} it suffices to show that the Dirichlet mean exit time ${}^{\D}\!E^{\bar \cb_{\gamma}(R)}$ of $\bar \cb_{\gamma}(R)$ is a finite function.
Indeed, let
\begin{equation}\label{function_E_horoball}
E_{R} = \frac{R -b_{\gamma}}{(m-1)A}.
\end{equation}
As above, the curvature condition implies that $\Delta E_R\leq -1$ on $\cb_{\gamma}(R)$. Since  $E_R\geq 0$ on $\bar \cb_{\gamma}(R)$, it follows from Lemma~\ref{UpBndMeanExitTime} that $ E_R$ is an upper bound for ${}^{\D}\!E^{\bar \cb_{\gamma}(R)}$ on $\bar \cb_{\gamma}(R)$, which is therefore finite.

Actually, since $\bar \cb_{\gamma}(R)$ is $\D$-parabolic, when $M$ is a hyperbolic space $\hh(-A^2)$, the function defined in \eqref{function_E_horoball} is precisely the Dirichlet mean exit time of the horoball.

Let $\complement{\cb_{\gamma}(R)} = \{x \in M : b_{\gamma}(x) \geq R\}$ be the complement of a half space in $M$. Despite of the fact that every half space (horoball) is $\D$-parabolic, this is not the case for its complement. In fact, a simple computation using the Hessian comparison theorem, \cite{BLPS}, shows that the function
$$ u(x) = e^{-A(m-1)R} - e^{-A(m-1)b_{\gamma}}$$
is a nonnegative bounded solution of
\begin{equation*}
\left\{ \begin{array}{rl}
\Delta u  \geq  0 & \text{ in }  \ \inte\,\complement{\cb_{\gamma}(R)}  \\
u = 0 & \text{ on }  \ \partial \complement{\cb_{\gamma}(R)}.
\end{array} \right.\\
\end{equation*}
By Proposition \ref{DirParSubharmonic} the manifold with boundary $\complement{\cb_{\gamma}(R)}$ is $\D$-hyperbolic. In particular, a solution, if any, of the problem
\begin{equation*}
\left\{ \begin{array}{rl}
\Delta E + 1  =  0 & \text{ in }  \ \inte\,\complement{\cb_{\gamma}(R)}  \\
E = 0 & \text{ on }  \ \partial \complement{\cb_{\gamma}(R)},
\end{array} \right.\\
\end{equation*}
may not be the mean exit time of $\complement{\cb_{\gamma}(R)}$ because there is no global uniqueness even when $E$ is bounded.

\begin{theorem}\label{th-negativecurv}
Let $(M,g)$ be a Cartan-Hadamard manifold satisfying $\mathrm{Ric} \geq -(m-1)B^2$ and $\mathrm{Sec} \leq -A^2$, for some constants $B \geq A >0$, in the complement $\complement \cb_{\gamma}(R)$ of a half space $\cb_{\gamma}(R)$ with respect to some geodesic ray $\gamma$. If $h$ is any Riemannian metric on $M$ satisfying $h =g$ on $\complement \cb_{\gamma}(R)$, then $(M,h)$ is $L^1$-Liouville.
\end{theorem}
\begin{proof}
Fix $R_k > R$ be a strictly increasing divergent sequence of positive constants and let
$$\mathrm{A}_k = \{x \in M : R \leq b_{\gamma}(x) \leq R_k\}$$
be the horoannulus contained in $\complement \cb_{\gamma}(R)$. Since $\mathrm{A}_k$ is a smooth subset of $\cb_{\gamma}(R_k)$ and every horoball is $\D$-parabolic, by Corollary \ref{D-parabolicity of subdomains} $\mathrm{A}_k$ is also $\D$-parabolic.
A straightforward application of comparison arguments (cf. \cite{BLPS}) shows that the function
$$ E_{k}(x) = - \frac{R_k -R}{(m-1)A}\cdot\frac{e^{-(m-1)Ab_{\gamma}(x)}-e^{-(m-1)AR_k}}{e^{-(m-1)AR}-e^{-(m-1)AR_k}} + \frac{R_k -b_{\gamma}(x)}{(m-1)B}$$
satisfies $\Delta E_{k}(x)\geq -1$ in $\mathrm{A}_{k}$. We claim that $ E_{k}(x)\leq {}^{\D}E^{\mathrm{A}_{k}}(x)$ where ${}^{\D}E^{\mathrm{A}_{k}}$ is the Dirichlet mean exit time of $\mathrm{A}_{k}$.

Indeed, define $\Omega = \{x \in \mathrm{A}_k : {}^{\D}E^{\mathrm{A}_{k}}(x) - E_{k}(x) < 0\}$ and consider $u(x) = E_{k}(x) - {}^{\D}E^{\mathrm{A}_{k}}(x)$. Is is clear that $u$ is bounded and subharmonic on $\Omega$. Being $\mathrm{A}_{k}$ $\D$-parabolic, by the Ahlfors characterization of $\D$-parabolicity (cf. Proposition \ref{DirParSubharmonic}),
$$ \sup_{\Omega} u = \sup_{\partial \Omega} u = 0.$$
Therefore, $E_{k}(x) \leq {}^{\D}E^{\mathrm{A}_{k}}(x)$ on $\mathrm{A}_{k}$.
Letting $k \to + \infty$ we have ${}^{\D}E^{\mathrm{A}_{k}}(x) \to + \infty$ and the conclusion follows by Corollary \ref{thm_comparison_dirichlet}.
\end{proof}

\begin{remark}\rm
We note that, when $A=B$ in the above theorem, the function
$$ E_{k}(x) = - \frac{R_k -R}{(m-1)A}\cdot\frac{e^{-(m-1)Ab_{\gamma}(x)}-e^{-(m-1)AR_k}}{e^{-(m-1)AR}-e^{-(m-1)AR_k}} + \frac{R_k -b_{\gamma}(x)}{(m-1)A},$$
is precisely the Dirichlet mean exit time of $\mathrm{A}_k$.
\end{remark}

\subsection{Manifolds containing large warped product cones}

We now consider the Dirichlet mean exit time of cones in warped products. Let $M=[0,+\infty)\times_{r}\Sigma$ be a warped product where $\Sigma$ is a closed smooth Riemannian manifold with dimension $(m-1)$. For any domain  $\Omega$ in $\Sigma$ with smooth boundary we denote by $C_\Omega$ the cone in $M$ over $\Omega$, namely   $C_\Omega = \{(r,\theta) \,: \, r\geq 0, \, \theta\in \Omega\}$. We denote by $\lambda_1(\Omega)$ the smallest Dirichlet eigenvalue of  $\Delta_\Sigma$ on $\Omega$. It is well known that $\lambda_1(\Omega)$ is simple and that its eigenfunctions have constant sign.

The next proposition shows that the fact that $C_\Omega$ is $\D$-$L^1$-Liouville depends on $\lambda_1(\Omega)$.
Similar computations and related results can be found in \cite[Theorem 3.1]{DevyverPinchoverPsaradakis}.

\begin{proposition}
\label{L1-Liouville-cones} Let $C_\Omega$ be the cone over $\Omega\subset \Sigma$ in the warped product $M=[0,+\infty)\times_{r}\Sigma$.
\begin{itemize}
\item[(i)] If $\lambda_1(\Omega)>2m$ then $C_\Omega$ has finite Dirichlet mean exit time (and it is therefore not Dirichlet $L^1$-Liouville);
\item[(ii)] If $\lambda_1 (\Omega)\leq 2m$ then $C_\Omega$ has infinite Dirichlet mean exit time.
\end{itemize}
\end{proposition}
\begin{proof}
Assume first that $\lambda_1(\Omega)>2m$ and let $\Omega'\supset \bar \Omega$ be such that $\lambda'=\lambda_1(\Omega')>2m$.
Let $u$ be an eigenfunction belonging to $\lambda'$. By what recalled above, we may choose $u$ to be strictly positive in $\Omega'$ and therefore $\min_{\bar\Omega}u =u_o>0$. We consider the function $\varphi(r,\theta)=r^2u(\theta)$. From the expression of the Laplacian of $M$,
\[
\Delta  = r^{-m+1} \frac{\partial}{\partial r} \left(
r^{m-1}\frac{\partial}{\partial r} \right) + \frac{1}{r^2}\Delta_\Sigma,
 \]
and using $u\geq u_o$ on $\Omega$, and $\lambda' >2m$ we deduce that
\[
\Delta\varphi = -(\lambda'- 2m)u(\theta)\leq -(\lambda' - 2m) u_o,
\]
so that there exists a constant $c$ such that
\[
\Delta (c\varphi) \leq -1.
\]
The conclusion follows from Lemma \ref{UpBndMeanExitTime}.

To prove (ii) we assume first that $\lambda=\lambda_1(\Omega)<2m$. For every $R>0$ we  let $\varphi_R(r,\theta) = h_R(r)u(\theta)$ where $u$ is a positive eigenfunction belonging to $\lambda$ normalized in such a way that
  $\max_\Omega u=u(\theta_o)=1,$ and
\[
h_R(r)= \frac{1}{2m-\lambda}\left[ \left(\frac{r}{R}\right)^{\alpha} R^2 - r^2\right],
\]
with
\[
\alpha = \frac{-(m-2) +\sqrt{(m-2)^2 +4\lambda}}2.
\]
Note that, since $\lambda<2m$, $\alpha<2$ and therefore $h_R(1)\to +\infty$ as $R\to +\infty$.

A straightforward computation shows that $h_R$ is a solution of the problem
\begin{equation*}
\begin{cases}
\displaystyle{h_R'' + \frac{m-1} r h'_R -  \frac{\lambda}{r^2}h_R= -1} &\\
h_R(0)=h_R(R)=0, \,\, h_R(r)>0 \,\text{ in  } \,(0,R), &
\end{cases}
\end{equation*}
so that, recalling the expression of the Laplacian of the warped product $M$ we see that $\varphi_R$ satisfies
\[
\Delta \varphi_R (r,\theta)= -u(\theta)\geq -1, \quad \varphi_R=0 \,\text{ on } \partial C_{\Omega, R}.
\]
By comparison, the mean exit time ${}^{\D}\!E^{C_{\Omega, R}}$ of the truncated cone $C_{\Omega,R}$ satisfies
\[
\varphi_R \leq {}^{\D}\!E^{C_{\Omega, R}}.
\]
Thus
\[
{}^{\D}\!E^{C_{\Omega, R}}(1,\theta_o)\geq \varphi_R(1,\theta_o)= h_R(1)\to +\infty \,\,\text{ as } \,\, R\to +\infty,
\]
showing that the Dirichlet mean exit time of   $C_\Omega$ is infinite.

Finally, suppose that $\lambda=\lambda_1(\Omega)=2m$. In this case we let $\varphi(r,\theta)=h_R(r)u(\theta)$ where
\[
h_R(r)= \frac 1{m+2}\left[\frac{R^{m+2}}{R^{m+2}-1}(r^2-r^{-m})\log R - r^2\log r\right], \quad 1\leq r\leq R.
\]
It is easily checked that
\[
h_R'' +\frac{m-1} r h'_R -  \frac{2m}{r^2}h_R= -1 \,\text{ in }\, 1\leq r\leq R,
\]
$h_R(1)=h_R(R)=0$ and $h_R(2)\to +\infty$ as $R\to +\infty$.

Arguing as above one shows that the mean exit time ${}^{\D}\!E^{C_{\Omega,1,R}}$ of the region $C_{\Omega,1,R}=\{(r,\theta)\,:\, 1<r<R, \, \, \theta\in \Omega\}$ satisfies ${}^{\D}\!E^{C_{\Omega,1,R}}(2,\theta_o)\to +\infty$ as $R\to +\infty$, and, again, $C_{\Omega}$ has infinite Dirichlet mean exit time.
\end{proof}

As a consequence of the above proposition we have the following

\begin{theorem}\label{th-cone}
Let $(M,g)$ be a $m$-dimensional smooth Riemannian manifold. Assume that there exists a region $M_0$ in $M$ isometric to a warped product cone $C_{\Omega}$.
If $\lambda_1(\Omega) \leq 2m$, then $M$ is $L^1$-Liouville.
\end{theorem}
\begin{proof}
By Proposition \ref{L1-Liouville-cones} $C_{\Omega}$ has infinite Dirichlet mean exit time, and therefore so does $M_0$. Thus $M$ is $L^1$-Liouville by Corollary \ref{thm_comparison_dirichlet}.
\end{proof}

\section{Appendix}\label{appendix}

In this appendix we give some remarks on $\mathcal{N}$-parabolicity in order to prove that, under the assumption of compact boundary, the $\D$-parabolicity is equivalent to the $\mathcal{N}$-parabolicity. We will also collect some useful facts about functions satisfying weak differential inequalities of the form
\[
-\int_{\inte\,M} \langle \nabla u,\nabla \rho\rangle\geq 0 \quad \forall \,\, 0\leq \rho \in C^\infty_c(\Omega),
\]
where $\Omega$ is a nonnecessarily relatively compact domain in a manifold $M$ with possibly nonempty boundary $\partial M$. Note that by an approximation argument (cf. \cite{IPS_Crelle}, or \cite[Corollary 7]{PV_Extension}), it is equivalent to assume that the above inequality holds for every  $0\leq \rho\in W^{1,2}_c(\Omega)$.

We begin with a version of the usual comparison principle.

\begin{proposition}\label{comparison_weak_Neumann}
Let $\Omega$ be a relatively compact domain in a manifold with boundary $M$.
\begin{itemize}
\item[(i)] If $u\in C^0(\bar \Omega) \cap W^{1,2}_{loc}(\inte\,\Omega)$ satisfies
$$
-\int_{\Omega} \langle\nabla u,\nabla\varphi\rangle \geq 0 \quad \forall \,\,0\leq \rho\in C^\infty_c(\inte\, \Omega),
$$
then
$$
\sup_\Omega u= \sup_{\partial \Omega} u.
$$
\item[(ii)] If $u\in C^0(\bar \Omega) \cap W^{1,2}_{loc}(\Omega)$ satisfies
$$
-\int_{\Omega} \langle\nabla u,\nabla\rho \rangle \geq 0 \quad \forall \,\,0\leq \rho\in C^\infty_c(\Omega),
$$
then
$$
\sup_\Omega u= \sup_{\partial_0 \Omega} u.
$$
\end{itemize}
\end{proposition}
\begin{proof}
Case (i) is standard. To prove (ii) assume by contradiction that
$\sup_\Omega u >\sup_{\partial_0 \Omega} u+2\epsilon$ and set
$$
w_\epsilon = \max\{ u-\sup_{\partial_0 \Omega} u -\epsilon, 0\}\in C^0(\bar \Omega)\cap W^{1,2}_{c}(\Omega), \quad \Omega_\epsilon= \{x\,:\, u(x)>\sup_{\partial_0 \Omega}u +\epsilon\}\neq \emptyset.
$$
Using $w_\epsilon$ as test function in the weak differential inequality yields
\[
\int_{\Omega_\epsilon}|\nabla u|^2 =0,
\]
and therefore $u$ is constant and equal to  $\sup_{\partial_0 \Omega} u+\epsilon$ on $\Omega_\epsilon$, contradicting the definition of $\Omega_\epsilon$.
\end{proof}

Recall now the following form of the strong maximum principle for weak solutions of $\Delta u\geq 0$, see e.g. \cite[Theorem 8.19, p. 198--199]{GT}, or \cite[Theorem 9]{StongMaxPrinciple}.

\begin{proposition}[Strong Maximum Principle]\label{strong_max_principle} Let $\Omega$ be a (not necessarily bounded) domain in a manifold $M$ and assume that $u\in C^0(\bar{\Omega})\cap W^{1,2}_{loc}(\inte\, \Omega)$ is bounded above and satisfies
\[
-\int_{\Omega} \langle \nabla u, \nabla \rho\rangle \geq 0 \quad \forall \,\,0\leq \rho \in C^\infty_c(\inte\, \Omega).
\]
If there exists $\bar x\in \inte\, \Omega$ such that $u(\bar x)=\sup_\Omega u$, then $u$ is constant in $\Omega$.
\end{proposition}

Indeed, according to \cite[Theorem 8.19]{GT} the conclusion holds even without assuming the continuity of $u$ provided that there exists a ball $B$ with $\bar B\subset \inte\, \Omega$ such that $\sup_B u = \sup_\Omega u$.

\begin{corollary}
\label{weak_bdry_point_lemma}
Let $\Omega$ be a domain in $M$ and assume that $u\in C^0(\bar{\Omega})\cap W^{1,2}_{loc}(\Omega)$ is bounded above and satisfies
\[
-\int_\Omega \langle \nabla u, \nabla \rho\rangle \geq 0 \quad \forall \,\,0\leq \rho \in C^\infty_c(\Omega), \ \ (\text{that is,} \ \ C^\infty_c(\bar{\Omega}\backslash \partial_0 \Omega)).
\]
If there exists $\bar x$ in the interion of $\partial_1 \Omega$ such that $u(\bar x)= \sup_\Omega u$, then $u$ is constant.
\end{corollary}
\begin{proof}
Indeed, assume that such a point $\bar x$ exists.
Take a neighborhood $B$ of $\bar x$ such that $B\cap \partial_0 \Omega=\emptyset$. According to
Proposition~\ref{comparison_weak_Neumann} (ii), $u(\bar x)= \sup_B u = \sup_{\partial_0 B} u$, so that there exists $x_o\in \inte\, \Omega$ such that $u(x_o)=\sup_\Omega u$ and $u$ is constant by the strong maximum principle above.
\end{proof}

The above result shows that if  $u\in C^0(\bar{\Omega})\cap W^{1,2}_{loc}(\Omega)$ is bounded above and satisfies
the differential inequality
$$
-\int_\Omega \langle\nabla u, \nabla \rho\rangle \geq 0 \quad \forall \,\,0\leq \rho \in C^\infty_c(\Omega),
$$
then $u$ cannot attain its supremum on $\partial_1 \Omega$, showing that the validity of the differential inequality for test functions which do not vanish on $\partial_1 \Omega$ is indeed the weak replacement of the
inequality $\partial u/\partial \nu\leq 0$ on $\partial_1 \Omega$ for $C^1$ functions, and that Corollary \ref{weak_bdry_point_lemma} may be viewed as a weak version of the boundary point lemma.

Using this we may also obtain the following characterization of Neumann parabolicity for manifolds $M$ with compact boundary $\partial M$.

\begin{proposition}
\label{cns_n-parabolicity} Let $M$ be a manifold with compact boundary $\partial M$. Then $M$ is $\mathcal{N}$-parabolic if and only if
\begin{itemize}
\item[(1)] for every $u\in C^0(M)\cap W^{1,2}_{loc}(\intM)$ which is bounded above and satisfies
\begin{equation}
\label{weak-subharmonic}
-\int_{\intM} \langle \nabla u, \nabla\rho \rangle \geq 0\quad \forall \,0\leq \rho\in C^\infty_c(\intM),
\end{equation}
we have
\[
\sup_M u = \sup_{\partial M} u.
\]
\end{itemize}
\end{proposition}
\begin{proof}
Assume that $M$ is $\mathcal{N}$-parabolic, and let $u$ be a function satisfying the condition listed in (1). Suppose by contradiction that $\sup_M u > \sup_{\partial M} u +2 \epsilon$, and let $v=\max\{ u-\sup_{\partial M} u+ \epsilon, 0\}$. Then $v$ satisfies \eqref{weak-subharmonic}, it is bounded above, and vanishes in a neighborhood of $\partial M$, so that, by the assumed $\mathcal{N}$-parabolicity of $M$, $v$ is constant and equal to $0$. Hence $u\leq \sup_{\partial M} u$ on $M$, contradiction.
To prove the reverse implication, assume that condition (1) holds and let $v\in C^0(M)\cap W^{1,2}_{loc}(M)$ be bounded above, and satisfy
\[
-\int_M \langle \nabla v, \nabla \rho\rangle \geq 0 \quad \forall \,\,0\leq \rho \in C^\infty_c(M).
\]
In particular, $v$ satisfies the differential inequality in (1), and it follows that $\sup_M v =\sup_{\partial M} v$. Since $\partial M$ is compact, there exists $\bar x \in \partial M$ such that $\sup_{\partial M} v= v(\bar x) =\sup_M v$, whence $v$ is constant by Corollary \ref{weak_bdry_point_lemma} and $M$ is $\mathcal{N}$-parabolic by definition.
\end{proof}

Indeed, the above result can be weakened by requiring that (1) holds for bounded harmonic functions.

\begin{corollary}\label{cns_n-parabolicity2}
Let $M$ be a manifold with compact boundary $\partial M$. Then $M$ is $\mathcal{N}$-parabolic if and only if for every $u\in C^0(M)$ bounded and satisfying $\Delta u=0$ in $\intM$ we have $\sup_M u = \sup_{\partial M} u$.
\end{corollary}
\begin{proof}
It suffices to show that if the condition in the statement holds, then $M$ is $\mathcal{N}$-parabolic. Suppose to the contrary that $M$ is not $\mathcal{N}$-parabolic, and therefore, by Proposition \ref{cns_n-parabolicity}, that there exists $v\in C^0(M)\cap W^{1,2}_{loc}(\intM)$ which is bounded above, satisfies
\[
-\int_{\intM} \langle \nabla v,  \nabla \rho\rangle \geq 0 \quad \forall \,\,0\leq \rho \in C^\infty_c(\inte\,M),
\]
and such that
\[
\sup_M v > \sup_{\partial M} v+ 2\epsilon.
\]
It follows that $w=\max\{v-\sup_{\partial M} v- \epsilon, 0\}\in C^0(M)\cap W^{1,2}_{loc}(\intM)
$ vanishes in a neighborhood of $\partial M$, is nonnegative, bounded above by $\mu=\sup_M v -\sup_{\partial M} v- \epsilon$ and subharmonic in $\inte\,M$.

Let $\{\Omega_k\}$ be an exhaustion of $M$ by relatively compact open sets with smooth boundary such that $\partial_1 \Omega_k =\partial M$ and $\bar{\Omega}_k \subset \Omega_{k+1}$ $\forall \,k$, and let $u_k$ be the solution of
\begin{equation*}
\left\lbrace
\begin{array}{rl}
\Delta u_k = 0 & \text{in } \inte\,\Omega_k\\
u_k=0 & \text{on } \partial_1 \Omega_k = \partial M\\
u_k=\mu & \text{on } \partial_0\Omega_k.
\end{array}\right.
\end{equation*}
By comparison,  $w\leq u_k$ and $u_k\geq u_{k+1}$ for every $k$, and therefore $u_k$ converges to a function $u$ which is harmonic in $\inte\,M$, satisfies $0\leq u\leq \mu$, vanishes on $\partial M$ and $u\geq w$. Thus
$\sup_M u\geq \sup_M w >0=\sup_{\partial M} u$, contradiction.
\end{proof}

As a simple consequence of Corollary \ref{cns_n-parabolicity2} we remark the fact that Neumann parabolicity is equivalent to Dirichlet parabolicity in the case of compact boundary. We note that the following result in particular applies to ends of a manifold $M$ without boundary.

\begin{corollary}\label{D vs N parabolicity bdr M compact}
Let $M$ be a manifold with compact boundary $\partial M$. Then $M$ is $\mathcal{N}$-parabolic if and only if it is $\mathcal{D}$-parabolic.
\end{corollary}

We conclude this appendix recalling the Khas'minskii test for Neumann parabolicity.
\begin{proposition}[Khas'minskii test]\label{Khas-N-parabolicity}
Let $M$ be a manifold with boundary $\partial M$. If there exist a compact set $K\subset M$ and a  function $0\leq \phi\in C^0(M\backslash \text{\r{K}})\cap W^{1,2}_{loc}(M\backslash K)$  such that $\phi(x)\to +\infty $ as $x$ diverges and
\[
-\int_{M\backslash K} \langle \nabla \phi,\nabla \rho\rangle \leq 0 \quad \forall \,\, 0\leq \rho\in C^0(M\backslash \text{\r{K}})\cap W^{1,2}_{loc}(M\backslash K),
\]
then $M$ is $\mathcal{N}$-parabolic.
\end{proposition}
\begin{proof}
It is a variation of the proof of the Khas'minskii test for $\mathcal{D}$-parabolicity. Let $u\in C^0(M)\cap W^{1,2}_{loc}(M)$ be bounded above and satisfy
\[
-\int_{M} \langle \nabla u, \nabla \rho\rangle \geq 0 \quad \forall \,\,0\leq \rho \in C^\infty_c(M),
\]
and assume by contradiction that $u$ is nonconstant. According to Proposition~\ref{comparison_weak_Neumann} and Corollary~\ref{weak_bdry_point_lemma}, $u$ cannot attain its supremum on $M$. Assuming without loss of generality that $\sup_M u>0$, choose  $0<\gamma <\sup_M u$ in such a way that $\sup_K u \leq \gamma$. Pick $x_o\in M\backslash K$ such that $u(x_o)>\gamma$, and set $v(x)= u-\gamma -\epsilon \phi$ where $\epsilon>0$ is  small enough that $v(x_o)>0$. Finally, let $\Omega=\{x\in M\backslash \text{\r{K}}\,:\, v(x)>0\}$. Since $\phi(x)\to +\infty$ as $x$ diverges, and $v(x)<0$ on $\partial K$, $\Omega$ is bounded, nonempty and  $\bar \Omega\cap K=\emptyset $. Since $v=0$ on $\partial \Omega$ and
\[
-\int_{\Omega} \langle \nabla v,\nabla \rho\rangle \geq 0 \quad \forall \,\, 0\leq \rho\in C_{c}^{\infty}(\Omega),
\]
by  Proposition~\ref{comparison_weak_Neumann} (ii), we have $v\leq 0$ on $\Omega$, contradiction.
\end{proof}

\subsection*{Acknowledgements} The authors would like to thank Luciano Mari for helpful conversations concerning the proof of Theorem \ref{thm_L1_nonnegative_curv}. They are also grateful to the anonymous referee for pointing out the relationships between Dirichlet parabolicity and criticality theory, for providing the proof in this framework of the existence of the Dirichlet Green's kernel and for suggesting several relevant references. Finally, they are grateful to Baptiste Devyver for having suggested the example in  Remark~\ref{KernelsVsParabolicity}.


\begin{thebibliography}{BKK}









\bibitem{ABR} S. Axler, P. Bourdon, W. Ramey, \textit{Harmonic function theory}.  Graduate Texts in Mathematics, \textbf{137}. Springer-Verlag, New York, 1992.

\bibitem{Bessa-Bar} G.P. Bessa, C. B\"ar, \textit{Stochastic completeness and volume growth.} Proc. Am. Math. Soc. \textbf{138}
(2010), 2629--2640.

\bibitem{BLPS} G.P. Bessa, J.H. de Lira, S. Pigola, A.G. Setti, \textit{Curvature estimates for submanifolds immersed into horoballs and horocylinders}. J. Math. Anal. Appl. \textbf{431} (2015), 1000--1007.

\bibitem{BPS-Pota} G.P. Bessa, S. Pigola, A.G. Setti, \textit{On the $L^1$-Liouville property of stochastically incomplete manifolds.} Potential Anal. \textbf{39} (2013), 313--324.

\bibitem{BS-preprint} D. Bianchi, A.G. Setti, \textit{Laplacian cut-offs, porous and fast diffusion on manifolds and other applications}. arXiv:1607.06008. Available at \url{https://arxiv.org/pdf/1607.06008}.

\bibitem {DePi}  B. Devyver, Y. Pinchover, \textit {Optimal $L^{p}$ Hardy-type inequalities}. Ann. Inst. H. Poincaré Anal. Non Lin\'eaire \textbf{33} (2016), 93--118.

\bibitem {DeMaPi} B.  Devyver, M. Fraas, Y. Pinchover, \textit{Optimal Hardy weight for second-order elliptic operator: an answer to a problem of Agmon}. J. Funct. Anal. \textbf{266} (2014), 4422--4489.

\bibitem{DevyverPinchoverPsaradakis} B. Devyver, Y. Pinchover, G. Psaradakis, \textit{
Optimal Hardy inequalities in cones}. arXiv:1502.05205. Available at \url{https://arxiv.org/pdf/1502.05205}.

\bibitem{Doob} J. L. Doob, \textit{Classical Potential Theory and Its Probabilistic Counterpart}, Reprint of the 1984 edition. Springer Verlag, New York 2001.

\bibitem{K} S. G. Krantz, \textit{Function Theory of Several Complex Variables}, 2nd edition. AMS Chelsea Publishing. American Mathematical Society 2001, Rhode Island.


\bibitem{GT} D. Gilbarg, N.S. Trudinger, \textit{Elliptic Partial Differential Equations of Second Order,} 2nd Edition. Springer-Verlag. Berlin Heidelberg 2001.

\bibitem{GP} V. Guillemin, D. Pollack, \textit{Differential topology}. Prentice-Hall, Inc., Englewood Cliffs, N.J., 1974.

\bibitem{Grigoryan_stochastic_harmonic} A. Grigor'yan, \textit{Stochastically complete manifolds and summable harmonic functions.} Izv. Akad. Nauk SSSR Ser. Mat. \textbf{52}, 1102--1108 (1988); translation in Math. USSR-Izv. 33, 425--432 (1989).

\bibitem{Grigoryan_laplace_eq} A. Grigor'yan, \textit{On the existence of positive fundamental solutions of the Laplace equation on Riemannian manifolds.} Mat. Sb. (N.S.) \textbf{128} (1985), no. 3, 354--363.

\bibitem{Grigoryan_Analytic} A. Grigor'yan, \textit{Analytic and geometric background of recurrence and non-explosion of the
Brownian motion on Riemannian manifolds.} Bull. Am. Math. Soc. (N.S.) \textbf{36} (1999), 135--249.

\bibitem{Grigoryan-Saloff-Coste} A. Grigor'yan and L. Saloff-Coste, \textit{Dirichlet heat kernel in the exterior of a compact set.} Comm. Pure. Appl. Math. \textbf{55} (2002), 93--133.

\bibitem{Gyrya_Saloff-Coste} P. Gyrya and L. Saloff-Coste, \textit{Neumann and Dirichlet Heat Kernels in Inner Uniform Domains.} Soci\'et\'e mathematique de France, 2011.

\bibitem{H} L. L. Helms, \textit{Potential Theory}, 2nd ed. Springer Verlag, London 2014.

\bibitem{It} S. It\^{o}, \textit{Martin boundary for linear elliptic differential operators of second order in a manifold}.
J. Math. Soc. Japan \textbf{16}(1964) 307--334.


\bibitem{Li} P. Li, \textit{Geometric analysis}. Cambridge Studies in Advanced Mathematics, \textbf{134}. Cambridge University Press, Cambridge, 2012.

\bibitem{IPS_Crelle} D. Impera, S. Pigola, A.G. Setti, \textit{Potential theory on manifolds with boundary and applications to controlled mean curvature graphs}. Crelle's Journal (to appear). DOI: 10.1515/crelle-2014-0137.

\bibitem{MP-book} W. Meeks, J. Pérez,  \textit{A survey on classical minimal surface theory.}
University Lecture Series, \textbf{60}. American Mathematical Society, Providence, RI, 2012.


\bibitem {Mu} M. Murata, \textit{Structure of positive solutions to $(-\Delta + V)u=0$ in $\rr^{n}$}. Duke Math. J. \textbf{53} (1986) 869--943.

\bibitem{Mu1}M. Murata, \textit{On construction of Martin boundaries for second order elliptic equations}. Publications of the Research Institute for Mathematical Sciences of Kyoto University \textbf{26} (1990), 585--627.

\bibitem{Perez-Lopez_MRSI}  F. Perez,  \textit{Parabolicity and minimal surfaces.} MSRI notes.

\bibitem{PS-JFA} S. Pigola, A.G. Setti, \textit{The Feller property on Riemannian manifolds}. J. Funct. Anal. \textbf{262} (2012), 2481--2515.

\bibitem{PRS_vanishing} S. Pigola, M. Rigoli and A.G. Setti, \textit{Vanishing and finiteness results in Geometric Analysis. A generalization of the Bochner technique.} Progress in Math. \textbf{266}, Birk\"auser, 2008.

\bibitem{StongMaxPrinciple} S. Pigola, A.G. Setti, \textit{The strong maximum principle: a brief survey of techniques}. In preparation.

\bibitem{PV_Extension} S. Pigola, G. Veronelli, \textit{The smooth Riemannian extension problem}. arXiv:1606.08320. Avaliable at \url{https://arxiv.org/pdf/1606.08320.pdf}.

\bibitem {Pi} Y. Pinchover, \textit{On positivity, criticality, and the spectral radius of the shuttle operator for elliptic operators}. Duke Math. J. \textbf{85} (1996) 431--445.

\bibitem {Pi1} Y. Pinchover, \textit{Criticality and ground states for second-order elliptic equations}. J. Differential Equations \textbf{80} (1989), 237--250.

\bibitem {Pi2} Y. Pinchover, \textit{On criticality and ground states of second-order elliptic equations II}, J. Differential Equations \textbf{87} (1990) 353--364.

\bibitem {Pinsky}  R.G. Pinsky, \textit{Positive harmonic functions and diffusion}. Cambridge Studies in Advanced Mathematics, \textbf{45}. Cambridge University Press, Cambridge, 1995.

\bibitem {PiSa} Y. Pinchover, T. Saadon, \textit{On positivity of solutions of degenerate boundary value problems for second-order elliptic equations}. Israel J. Math. \textbf{132} (2002), 125--168.


\bibitem{Saloff-Coste_book} L. Saloff-Coste, \textit{Aspects of Sobolev-type inequalities.} London Math. Soc. Lectures Series \textbf{289}, Cambridge Univ. press., Cambridge 2002.

\bibitem {Si} B. Simon, \textit{Large time behavior of the $L^p$ norm of Schrödinger semigroups}. J. Funct. Anal. \textbf{40} (1981), 66--83.

\bibitem{Tr-Siberian} M. Troyanov, \textit{Parabolicity of manifolds}. Siberian Adv. Math. \textbf{9} (1999), 125--150.

\end{thebibliography}
\end{document}